\documentclass{amsart}
\usepackage{amsmath,amsthm,amssymb}
\usepackage{amsrefs}
\usepackage{proof}
\usepackage{url}
\usepackage{hyperref}
\newtheorem{proposition}{Proposition}
\newtheorem{lemma}{Lemma}
\newtheorem{theorem}{Theorem}

\newtheorem{corollary}{Corollary}
\newtheorem{claim}{Claim}
\theoremstyle{definition}
\newtheorem{definition}{Definition}
\theoremstyle{remark}
\newtheorem*{ack}{Acknowledgment}

\newcommand{\APPROXEQ}[0]{\trianglelefteq}
\newcommand{\EMPTYSEQ}[0]{[\ ]}

\newcommand{\minusdot}[0]{
  \mathrel{\vcenter{\offinterlineskip\ialign{%
      \hfil##\hfil\cr%
      \noalign{\kern1ex}
      $-$\cr
      \noalign{\kern-.3ex}
      $\dot{}$\cr
    }}}
  }
\newcommand{\Multi}[1]{\vec{#1}}
\newcommand{\goedel}[1]{\lceil {#1} \rceil}
\newcommand{\divtwo}[1]{\lfloor \frac{#1}{2} \rfloor}
\newcommand{\COMPATIBLE}[0]{\bigtriangleup}

\DeclareMathOperator{\PIND}{PIND}
\DeclareMathOperator{\Env}{Env}

\DeclareMathOperator{\Fun}{Fun}

\DeclareMathOperator{\PV}{\textup{PV}}

\DeclareMathOperator{\Cond}{\textup{Cond}}
\DeclareMathOperator{\proj}{\textup{proj}}

\DeclareMathOperator{\Base}{\textup{Base}}
\DeclareMathOperator{\ar}{\textup{ar}}
\DeclareMathOperator{\Eq}{\textup{Eq}}

\newcommand{\CompSt}[3]{\langle {#1}, {#2} \rangle \downarrow {#3}}

\newcommand{\Bracket}[2]{\langle {#1}, {#2} \rangle}
\DeclareMathOperator{\dom}{\mathrm{dom}}

\newcommand{\emptyseq}{{\varepsilon}}

\DeclareMathOperator{\pair}{\mathrm{pair}}

\newcommand{\BLEN}[1]{\lvert {#1} \rvert}
\newcommand{\SIZE}[1]{\mathrm{size}({#1})}
\newcommand{\STEPS}[1]{\mathrm{nodes}({#1})}

\begin{document}

\title[Consistency proof of a fragment of $\PV$ with substitution]{Consistency proof of a fragment of $\PV$ with substitution in bounded arithmetic}

\author{Yoriyuki Yamagata}
\address{National Institute of Advanced Science and Technology (AIST) \\
1-8-31 Midorigaoka, Ikeda, Osaka 563-8577 Japan}
\urladdr{https://staff.aist.go.jp/yoriyuki.yamagata/en/}
\email{yoriyuki.yamagata@aist.go.jp}


\begin{abstract}
  This paper presents proof that Buss's $S^2_2$ can prove the consistency of a fragment of Cook and Urquhart's $\PV$ from which induction has been removed but substitution has been retained.
  This result improves Beckmann's result, which proves the consistency of such a system without substitution in bounded arithmetic $S^1_2$.

  Our proof relies on the notion of ``computation'' of the terms of $\PV$.
  In our work, we first prove that, in the system under consideration, if an equation is proved and either its left- or right-hand side is computed, then there is a corresponding computation for its right- or left-hand side, respectively.
  By carefully computing the bound of the size of the computation, the proof of this theorem inside a bounded arithmetic is obtained, from which the consistency of the system is readily proven.

  This result apparently implies the separation of bounded arithmetic because Buss and Ignjatovi\'c stated that it is not possible to prove the consistency of a fragment of $\PV$ without induction but with substitution in Buss's $S^1_2$.
  However, their proof actually shows that it is not possible to prove the consistency of the system, which is obtained by the addition of propositional logic and other axioms to a system such as ours.
  On the other hand, the system that we have considered is strictly equational, which is a property on which our proof relies.
\end{abstract}

\maketitle

\section{Introduction}
\label{sec:intro}

Ever since Buss showed the relation between his hierarchy of bounded arithmetic, $S^i_2, i = 1, 2, \ldots$, and the polynomial time hierarchy of computational complexity \cite{Buss}, the question of whether his hierarchy collapses at some $i=n$ has become a central question in bounded arithmetic.
This is because the collapse of Buss's hierarchy implies the collapse of polynomial time hierarchy.

A classical way to prove the separation of theories is to use the second incompleteness theorem of G\"odel.
For example, if it is proved that $S_2$ proves the consistency of $S^1_2$, $S^1_2 \not= S_2$ is obtained, because $S^1_2$ cannot prove its own consistency.

Wilkie and Paris showed that $S_2$ cannot prove the consistency of Robinson arithmetic $Q$ \cite{Wilkie1987Scheme}, which is a much weaker system.
Although this result stems more from the free use of unbounded quantifiers than from the power of arithmetic, Pudl\'ak showed that $S_2$ cannot prove the consistency of bounded proofs (proofs in which the formulas only have bounded quantifiers) of $S^1_2$ \cite{PudlakNote}.
The result was refined by Takeuti \cite{takeuti1990some}, as well as by Buss and Ignjatovi\'c \cite{BussUnprovability}, who showed that, even if induction is removed from $S^1_2$, $S_2$ is still not able to prove the consistency of its bounded proofs.

Thus, it will be interesting to delineate theories that can be proven to be consistent in $S_2$ and $S^1_2$ in order to find a theory $T$ that can be proven to be consistent in $S_2$ but not in $S^1_2$.
In particular, we focus on Cook and Urquhart's system $\PV$~\cite{cook1975feasibly}, which is essentially an equational version of $S^1_2$.
Buss and Ignjatovi\'c stated that $\PV$ cannot prove the consistency of $\PV^-$, a system based on $\PV$ from which induction has been removed but substitution is retained.
On the other hand, Beckmann \cite{Beckmann2002} later proved that $S^1_2$ can prove the consistency of a theory obtained from $\PV^-$ by removing the substitution rule.

This paper presents proof that $S^2_2$ is capable of proving the consistency of purely equational $\PV^-$, in which proofs are formulated as trees.
This result apparently implies that $S^1_2 \subsetneq S^2_2$ is based on the result of Buss and Ignjatovi\'c.
However, their proof actually shows that $\PV$ cannot prove the consistency of the extension of $\PV^-$ that contains propositional logic and $BASIC^e$ axioms.
On the other hand, our $\PV^-$ is strictly equational, which is a property on which our proof relies.
Although Buss and Ignjatovi\'c stated that their proof can be extended to purely equational $\PV^-$, there is a gap in their reasoning.
We discuss this in Section \ref{sec:buss-ignja} in detail.

The consistency of $\PV^-$ can be proven by using the following strategy.
Beckmann used a rewriting system to prove the consistency of $\PV^-$ by excluding the substitution rule.
According to the terminology of programming language theory, the use of a rewriting system to define the evaluation of terms is known as \emph{small-step semantics} (referred to as \emph{structural operational semantics} in \cite{Plotkin1981}).

There is an alternative approach toward obtaining the abovementioned definition, namely, \emph{big-step semantics} (referred to as \emph{natural semantics} in \cite{Kahn1987}).
In big-step semantics, the relation $\CompSt{t}{\rho}{v}$, where $t$ is a term, $\rho$ is an assignment to free variables in $t$, and $v$ is the value of $t$ under assignment $\rho$, is defined.
We treat $\CompSt{t}{\rho}{v}$ as a statement in a derivation and provide rules for deriving $\CompSt{t}{\rho}{v}$.
For technical reasons, it is assumed that such derivations are directed acyclic graphs (DAGs) in this paper.

However, it is still not possible to prove the induction step for the substitution rule, because bounded arithmetic cannot prove the existence of a value for each term of $\PV$.
We overcome this difficulty by allowing an approximate value of a computation, in a way similar to that described in Beckmann's paper \cite{Beckmann2002}.

Then, we attempt to prove that $\CompSt{t}{\rho}{v}$ implies that $\CompSt{u}{\rho}{v}$ for any given assignment $\rho$ by induction on the construction of the proof $\chi$ of $t = u$.
We call this fact soundness (with respect to our computational semantics).
It is possible to set bounds for all quantifiers that appear in the induction hypothesis of this induction by setting a bound on the G\"odel number of $\rho$ and bounds on the G\"odel numbers of the derivation of $\CompSt{t}{\rho}{v}$ and $\CompSt{u}{\rho}{v}$.
Because induction is carried out on bounded formulas, the proof can be carried out inside $S^i_2$ for some $i$.
Let the number of primitive symbols in $a$ $\SIZE{a}$.
We can show that $\SIZE{\rho}$ is polynomially bounded by $\SIZE{\chi}$.

The bounds for the derivations are more difficult to obtain.
Although it is possible to bound the number of \emph{nodes} in the above-mentioned derivations, bounds for the \emph{G\"odel numbers} of these derivations are not trivially obtained, because there are no (obvious) bounds for the terms that appear in the derivations.
This difficulty is overcome by employing the \emph{call-by-value} style of big-step semantics, in which a derivation has the form
\begin{equation}
  \infer{\CompSt{f(\Multi{t})}{\rho}{v}.}{
    \CompSt{f_1(\Multi{x})}{\nu_1}{w_1}, &
    \ldots &
    \CompSt{f_k(\Multi{x},y_1, \ldots, y_{l-1})}{\nu_l}{w_l},&
    (\CompSt{t_i}{\rho}{v_i})_{i = 1, \ldots, m}
    }
\end{equation}
where $\nu_j$ denotes the environment that maps $x_i$ to $v_i$ and $y_k$ to $w_k$ for $i = 1, \ldots, m$ and $k = 1, \ldots, j-1$.
$m$ is the number of the arguments of $f$.
Because the numbers of symbols in $t_1, \dots, t_m$ and $f_1, \ldots, f_l$ are bounded by $\SIZE{f(\Multi{t})}$, and the size of the values appearing in the derivation can be proven to be polynomially bounded by the number of nodes in the derivation and the size of conclusions, the size of the terms that appear in this derivation can be polynomially bounded by the number of nodes and size of the conclusions of the derivation.
Thus, all the quantifiers in the induction hypothesis are bounded by the G\"odel number of $\chi$.

The part of the induction step that is most difficult to prove is the soundness of the substitution rule.
The proof is divided into two parts.
First, it is proven that if $\sigma$ derives $\CompSt{t_1[u/x]}{\rho}{v_1}, \ldots, \CompSt{t_n[u/x]}{\rho}{v_n}$ and contains a computation of $\CompSt{t}{\rho}{v}$, then there exists $\tau$ that derives $\CompSt{t_1}{\rho[x \mapsto v]}{v_1}, \ldots, \CompSt{t_n}{\rho[x \mapsto v]}{v_n}$ (Substitution I).
Next, it is proven that if $\sigma$ derives $\CompSt{t_1}{\rho[x \mapsto v]}{v_1}, \ldots, \CompSt{t_n}{\rho[x \mapsto v]}{v_n}$ and contains a computation of $\CompSt{t}{\rho}{v}$, then there exists $\tau$ that derives $\CompSt{t_1[u/x]}{\rho}{v_1}, \ldots, \CompSt{t_n[u/x]}{\rho}{v_n}$ (Substitution II).

The intuition underlying the proof of Substitution I is explained as follows.
The na\"ive method, which uses induction on the length of $\sigma$, is ineffective.
This is because an assumption of the last inference of $\sigma$ may be used as an assumption of another inference; thus, it may not be a conclusion of $\sigma_1$, which is obtained from $\sigma$ by removing the last inference.
Therefore, it is not possible to apply the induction hypothesis to $\sigma_1$.
To transform all the assumptions into conclusions, it is necessary to increase the length of $\sigma_1$ from $\sigma$ by duplicating the inferences from which the assumptions are derived.
Therefore, induction cannot be used on the length of $\sigma$.

Instead, we use induction on $\SIZE{t_1[\emptyseq/x]} + \cdots + \SIZE{t_n[\emptyseq/x]}$ where $\emptyseq$ is a constant symbol.
Then, we prove that for all $\STEPS{\sigma} \leq U - \SIZE{t_1[\emptyseq/x]} - \cdots - \SIZE{t_n[\emptyseq/x]}$, where $U$ is a large integer that is fixed during the proof of soundness, we have $\tau$, which derives $\CompSt{t_1}{\rho[x \mapsto w]}{v_1}, \ldots, \CompSt{t_n}{\rho[x \mapsto w]}{v_n}$ and satisfies $\STEPS{\tau} \leq \STEPS{\sigma} + \SIZE{t_1[\emptyseq/x]} + \cdots + \SIZE{t_n[\emptyseq/x]}$ where $w$ is a value of $u$.
Because all the quantifiers are bounded, the proof can be carried out in $S_2$, in particular $S^2_2$.

This paper is a revised version of the paper titled ``Consistency proof of a feasible arithmetic inside a bounded arithmetic,'' ~\cite{Yamagata2014} which was posted to ArXiv.
It is revised from two aspects.
First, it addresses the problem in the proof that causes Beckmann's counter-example.
Second, it strengthens the meta-theory from $S^1_2$ to $S^2_2$, which is used to prove consistency.
$S^2_2$ is necessary to prove the soundness of transitivity and substitution rules
We discuss this point in Section \ref{subsec:meta}.

The remainder of this paper is organized as follows.
Section \ref{sec:pre} summarizes the preliminaries.
Section \ref{sec:PV} introduces $\PV$ and $\PV^-$, which is the target of our consistency proof.
Section \ref{sec:comp} introduces the notion of \emph{(approximate) computation}.
Section \ref{sec:estimate} shows that for each computation $\sigma$, $\SIZE{\sigma}$ is polynomially bounded by the number of nodes in $\sigma$ and the number of primitive symbols in the conclusion of $\sigma$.
Section \ref{sec:basics} presents technical lemmas that are used in the consistency proof.
Section \ref{sec:consistency} presents the proofs of the consistency of $\PV^-$.
Finally, Section \ref{sec:discuss} concludes the paper with a brief discussion.

\section{Preliminary}\label{sec:pre}

The sequence $a_1, a_2, \ldots, a_n$ is often abbreviated as $\Multi{a}$.
If we treat the sequence $a_1, a_2, \ldots, a_n$ as a single object, we write $[a_1, a_2, \ldots, a_n]$.
For each sequence $a$, $(a)_i$ is its $i$-th element.
We denote an empty sequence by $\EMPTYSEQ$ in the meta-language.
For integer $n$, $|n|$ denotes its length of binary representation.
For a set $A$ of integers, $\sum A$ denotes the sum of all members of $A$.

Many types of objects are considered as proofs of $\PV$, terms of $\PV$, or the \emph{computation} of these terms, all of which require the assignment of G\"odel numbers to them.
As all the objects under consideration can be coded as finite sequences of primitive symbols, it will suffice to encode these sequences of symbols.
Variables $x_1, x_2, \ldots$ are encoded by variable names $x$ and natural numbers $1, 2, \ldots$, which can be represented by binary strings.
Function symbols for all polynomial time functions are encoded using trees of the primitive functions and labels that show how the function is derived using Cobham's inductive definition of polynomial time functions.
Thus, the symbols that are used in our systems are finite, which enables us to use the numbers $0, \ldots, N$ to code these symbols.
Then, the sequence of symbols is coded as $N+1$-adic numbers.

For each object $a$ consisting of symbols, $\SIZE{a}$ denotes the number of primitive symbols in $a$, that is, the number of $N+1$-adic numbers in its G\"odel number.
If $a$ is a sequence or tree in an object language, $\STEPS{a}$ denotes the number of nodes in $a$.

We use the notation $a \equiv b$ when $a$ and $b$ are syntactically equivalent.

For a given term $t$, the notion of subterm $u$ is defined in the usual way.
Further, $u$ may be identical to $t$.
If $t \not\equiv u$, we call $u$ a proper subterm.

\section{PV and related systems}\label{sec:PV}

In this section, we introduce our version of $\PV$ and $\PV^-$.

$\PV$ is formulated as a theory of binary strings rather than integers.
We identify binary strings and integers which are represented in little endian (the least significant bit appears at the right most position).
The differences between our version of $\PV$ and the original $\PV$ are discussed in Section \ref{subsec:original-pv}.

$\PV$ provides the symbols for the empty sequence $\emptyseq$ and its binary successors $0, 1$ denoted by $b, b_1, \ldots$, which add 0 or 1 to the leftmost positions of strings.
If a term is solely constructed by $\emptyseq, 0, 1$, it is referred to as a \emph{numeral}.
Although binary successors are functions, the notation we use for them employs a special convention to omit the parentheses after the function symbol.
Thus, we write $01x$ instead of $0(1(x))$.

The language of $\PV$ contains function symbols for all polynomial time functions.
In particular, it contains the constant $\emptyseq$, binary successors $0, 1$ and $\emptyseq^n$, $\proj^i_n$.
The intuitive meaning is that $\emptyseq^n$ is the $n$-ary constant function whose value is $\emptyseq$, and $\proj^k_n$ is the projection function.
From here, a function symbol for a polynomial time function $f$ and $f$ itself are often identified.
The terms are denoted by $t, t_1, \ldots, u, r, s, \ldots$

For each function symbol $f$ of a polynomial time function, let $\Base(f)$ be the set of function symbols that are used in Cobham's recursive definition of $f$.
We assume that $\Base(f)$ always contains $\emptyseq$, $0$, and $1$, regardless of $f$.
For a set of function symbols $S$, we define $\Base(S) = \bigcup_{f \in S} \Base(f)$.
If $\alpha$ represents any sequence of symbols, $\Base(\alpha)$ is defined by the union of $\Base(f)$ for the function symbols $f$ that appear in $\alpha$.
$\Base(f)$ is computable by a polynomial time function.
For each function symbol $f$, $\ar(f)$ is the arity of $f$.
We encode $f \equiv \emptyseq, 0, 1, \emptyseq^n, \proj^i_n$ by
\begin{align}
  \goedel{\emptyseq} &= \goedel{[\Fun, \emptyseq]}\\
  \goedel{0} &= \goedel{[\Fun, 0]}\\
  \goedel{1} &= \goedel{[\Fun, 1]}\\
  \goedel{\emptyseq^n} &= \goedel{[\Fun, \emptyseq, \overbrace{\# \cdots \#}^n]}\\
  \goedel{\proj^i_n} &= \goedel{[\Fun, \text{proj}, i, \overbrace{\# \cdots \#}^n]}
\end{align}
where $\#$ is a ``filler'' symbol.
Then, a function defined by the composition of $g$ and $h_1, \ldots, h_m$ is encoded as $[\Fun, \mathrm{comp}, g, h_1, \ldots, h_n]$.
A function defined by the recurrence of $g_\emptyseq, g_0, g_1$ is encoded as $[\Fun, \mathrm{rec}, g_\emptyseq, g_0, g_1]$.
Then, for any function symbols $f$, which are defined by Cobham's inductive definition, $\ar(f) \leq \SIZE{f}$ is satisfied.

The only predicate in the vocabulary of $\PV$ is the equality $=$.
Our $\PV$ does not have inequalities $\leq, \geq, \ldots$.
The formulas $t_1 = t_2$ of $\PV$ are formed by connecting two terms $t_1, t_2$ by the equality $=$.
We consider $\PV$ to be purely equational; hence, the formulas do not contain propositional connectives and quantifiers.

There are three types of axioms and inferences in $\PV$: defining axioms, equality axioms, and induction.
We consider that the proofs in $\PV$ are all tree-like, not DAG-like.
This restriction to the representation of proofs is essential to our consistency proof.

\subsection{Defining axioms}\label{subsec:defn}
For all of Cobham's defining equations of polynomial time functions, there are corresponding defining axioms in $\PV$.
For the constant function $\emptyseq^n$, the defining axiom is
\begin{equation}
  \emptyseq^n(x_1, \ldots, x_n) =  \emptyseq
\end{equation}
for a positive integer $n$.
For the projection function, the defining axiom is
\begin{equation}
  \proj^i_n(x_1, \ldots, x_n) = x_i
\end{equation}
for a positive integer $n$ and an integer $i, 1 \leq i \leq n$.
For the binary successor functions $0$ and $1$, there is no defining axiom.
If the function $f(x_1, \ldots, x_n)$ is defined by the composition of $g$ and $h_1, \ldots, h_m$, the defining axiom is
\begin{equation}
  f(x_1, \ldots, x_n) = g(h_1(x_1, \ldots, x_n), \ldots, h_m(x_1, \ldots, x_n)).
\end{equation}
For the function defined by recursion of binary strings, there are three defining axioms:
\begin{align}
  f(\emptyseq, x_1, \cdots, x_n) &= g_\emptyseq(x_1, \cdots, x_n)\\
  f(0x, x_1, \cdots, x_n) &= g_0(x, f(x, \Multi{x}), x_1, \cdots, x_n)\\
  f(1x, x_1, \cdots, x_n) &= g_1(x, f(x, \Multi{x}), x_1, \cdots, x_n).
\end{align}
Using Cobham's recursive definition of polynomial time functions, it is easy to see that all polynomial time functions can be defined using these defining axioms.
Even though Cook and Urquhart's $\PV$~\cite{cook1975feasibly} requires all recursion schema to be bounded by a function with a polynomial growth rate, we do not impose this restriction.
Thus, our theory can be extended beyond polynomial time functions.
However, this paper focuses on the theory based on polynomial time functions.

We present defining axioms of forms $f(\Multi{x}) = t$, but we also introduce defining axioms of forms $t = f(\Multi{x})$.

\subsection{Equality axioms}\label{subsec:eq}
The identity axiom is formulated as
\begin{equation}
  t = t
\end{equation}

The remaining equality axioms are formulated as inference rules rather than axioms.
\begin{equation}
  \infer{t = u}{u = t}
\end{equation}
\begin{equation}
  \infer{t = r}{t = u & u = r}
\end{equation}
\begin{equation}
  \infer{f(t_1, \ldots, t_n) = f(u_1, \ldots, u_n)}{
    t_1 = u_1 & \cdots & t_n = u_n
  }
\end{equation}
\begin{equation}
  \infer{t(r) = u(r)}{t(x) = u(x)}
\end{equation}
for any term $r$.

\subsection{Induction}\label{subsec:ind}
\begin{equation}
  \infer{t_1(x) = t_2(x)}{
    t_1(\emptyseq) = t_2(\emptyseq) &
    t_1(s_ix) = v_i(t_1(x)) &
    t_2(s_ix) = v_i(t_2(x)) & (i = 0, 1)
  }
\end{equation}
The system $\PV$ contains defining axioms, equality axioms, and induction as axioms and inference rules.
In contrast, the system $\PV^-$ contains only defining axioms and equality axioms as axioms and inference rules.
This paper demonstrates that the consistency of $\PV^-$ can be proven by $S^2_2$.

\section{Approximate computation}\label{sec:comp}

In this section, we define the notion of \emph{approximate computations} as being the representation of the evaluations of the terms of $\PV$.
The idea that the computation values can be approximated using the $*$ symbol comes from Beckmann~\cite{Beckmann2002} but we only allow $*$ contained in numerals.

\begin{definition}[Approximate values]
  Let $\mathbb D$ be terms that are created by $0, 1$ from constants $\emptyseq$, $*$.
  $*$ stands for the \emph{unknown} value.
  The elements of $\mathbb D$ are called \emph{g-numerals}.
  For $v \in \mathbb D$, $\STEPS{v}$ is defined as the number of symbols $*, \emptyseq, 0, 1$.
  $\mathbb D$ has an order structure.
  For any $v, w \in \mathbb D$, let $\APPROXEQ$ be the relation that is recursively defined by
  \begin{enumerate}
     \item $\emptyseq \APPROXEQ \emptyseq$
     \item $v \APPROXEQ *$
     \item $v \APPROXEQ w \implies bv \APPROXEQ bw$ for $i = 0, 1$.
  \end{enumerate}
  If $v \APPROXEQ w$, $w$ is often written as $v^*$.
\end{definition}

\begin{lemma}[$S^1_2$]\label{lem:order}
  $\APPROXEQ$ in $\mathbb D$ is the order relation.
\end{lemma}

\begin{proof}
  \emph{Transitivity law} : we prove that f $v \APPROXEQ w$ and $w \APPROXEQ z$, $v \APPROXEQ z$ by induction on $\STEPS{v}+\STEPS{w}+\STEPS{z}$.
  If $z \equiv *$ then the conclusion follows.
  If $v \equiv *$ then $w \equiv *$ and $z \equiv *$ must hold.
  Therefore, $v \APPROXEQ z$.
  Next, if $z \equiv \emptyseq$, then $v \equiv w \equiv \emptyseq$.
  Thus, $v \APPROXEQ z$.
  If $v \equiv \emptyseq$ then $z$ must be either $*$ or $\emptyseq$.
  For both cases, $v \APPROXEQ z$.
  Finally, $v \equiv b_1v'$, $w \equiv b_2w'$ and $z \equiv b_3z'$.
  Then, $b_1 \equiv b_2 \equiv b_3$, $v' \APPROXEQ w'$ and $z \APPROXEQ z'$ hold.
  By induction hypothesis, $v' \APPROXEQ z'$ is therefore $v \APPROXEQ z$.
  Other cases are trivial.

  \emph{Anti-symmetry law} : we prove f $v \APPROXEQ w$ and $w \APPROXEQ v$, $v \equiv w$ by induction on $\STEPS{v}+\STEPS{w}$.
  If $v \equiv *$, then $w$ must be $*$; therefore, $v \equiv w$.
  Similarly, if $v \equiv \emptyseq$, then $w$ must be $\emptyseq$; therefore, $v \equiv w$.
  By a symmetric argument, we can assume that $v$ and $w$ are neither $*$ nor $\emptyseq$.
  Then, $v \equiv bv'$ and $w \equiv bw'$.
  $v' \APPROXEQ w'$ and $w' \APPROXEQ v'$ must hold.
  Therefore, $v' \equiv w'$.
  Thus, $v \equiv w'$.
\end{proof}

\begin{definition}
  Let $v_1, \ldots, v_n$ be g-numerals and $t$ be a term of $\PV^-$ with free variables $x_1, \ldots, x_n$.
  An \emph{environment} $\rho$ of $t$ is a map from $x_1, \ldots, x_n$ to $v_1, \ldots, v_n$ respectively.
  Let $\dom(\rho)$ be $\{x_1, \ldots, x_n\}$.
  Let
  \begin{equation}
    B(\rho) = \max_{i = 1}^m \STEPS{\rho(x_i)},
  \end{equation}
  $L(\rho) = n$ and $S(\rho) = \SIZE{\rho_i}$.
  The empty environment is denoted by $\EMPTYSEQ$.

  Let $v$ be a g-numeral, $t$ be a term of $\PV^-$, and $\rho$ be an environment of $t$.
  The form $\CompSt{t}{\rho}{v}$ is referred to as a \emph{(computation) judgment}, $t$ as the main term, $\rho$ as the environment, and $v$ as the value (of $t$).
  Because we allow approximate computations, a term $t$ may have several g-numerals as values under the same environment.

  If a computational judgment $\CompSt{t}{\rho}{v}$ has a form $\CompSt{f(x_1, \ldots, x_n)}{\rho}{w}$ or $\CompSt{v}{\rho}{w}$ where $v$ is a numeral, then it is called \emph{purely numerical}.
  An inference of a computational judgment is also called purely numerical if its conclusion and premises are purely numerical.
\end{definition}
A computation judgment can be derived using the following rules.
Each rule is attached by a symbol such as $*$ called a \emph{label}.

In the following rules, on the contrary to the case of terms $t(t_1, \ldots, t_n)$, $f(t_1, \ldots, t_n)$ for any function symbol $f$ means that $t_1, \ldots, t_n$ really appears in $f(t_1, \ldots, t_n)$.

  \begin{equation}\label{eq:comp-*}
    \infer[*]{\CompSt{t}{\rho}{*}}{
    }
  \end{equation}
  for any term $t$.

  \begin{equation}
    \label{eq:env}
    \infer[\Env]{\CompSt{x}{\rho[x \mapsto v]}{v^*}}{
    }
  \end{equation}
  where $v \APPROXEQ v^*$.

  \begin{equation}\label{eq:comp-0}
    \infer[v]{\CompSt{v}{\rho}{v^*}}{
    }
  \end{equation}
  where $v$ is a numeral and $v^*$ is an approximation of $v$.

  \begin{equation}\label{eq:comp-succ}
    \infer[b]{\CompSt{bt}{\rho}{bv^*}}{
      \CompSt{t}{\rho}{v}}
  \end{equation}
  where $b$ is either $0$ or $1$, $v^*$ is an approximation of $v$, $v$ is a g-numeral and $t$ is not a numeral.

  \begin{equation}\label{eq:comp-0-fun}
    \infer[\emptyseq^m]{\CompSt{\emptyseq^m(t_1, \ldots, t_m)}{\rho}{\emptyseq}}{
      (\CompSt{t_i}{\rho}{v_i})_{i=1, \ldots, m}
    }
  \end{equation}
  where $\emptyseq^m$ is the $m$-ary constant function of which the value is
  always $\emptyseq$.
  $(\CompSt{t_i}{\rho}{v_i})_{i=1, \ldots, m}$ is the sequence $\CompSt{t_{i_1}}{\rho}{v_1}, \ldots, \CompSt{t_{i_k}}{\rho}{v_m}$ of judgments.
  We use the similar notation from here.

  \begin{equation}\label{eq:comp-proj}
    \infer[\proj^i_m]{\CompSt{\proj_m^i(t_1, \ldots, t_m)}{\rho}{v^*_i}}{
      (\CompSt{t_j}{\rho}{v_j})_{j = 1, \ldots, m}
    }
  \end{equation}
  for $i = 1, \cdots, m$.
  $v^*_i$ is an approximation of $v_i$.

  If $f$ is defined by composition, we have the following rule.
  \begin{equation}\label{eq:comp-comp}
    \infer[\text{comp}]{\CompSt{f(t_1, \ldots, t_n)}{\rho}{z^*}}{
      \CompSt{g(\Multi{y}}{\xi}{z}
      &
      \CompSt{h_1(\Multi{x})}{\nu}{w_1}
       \cdots
      \CompSt{h_m(\Multi{x})}{\nu}{w_m} &
      (\CompSt{t_i}{\rho}{v_i})_{i = 1, \ldots, n}
    }
  \end{equation}
  where $\Multi{y} = y_1, \ldots, y_m$, $\Multi{x} = x_1, \ldots, x_n$, $\nu(x_i) = v_i, i = 1, \ldots, n$ and $\xi(y_j) = w_j, j = 1, \ldots, m$.

  If $f$ is defined by recursion, we have the following rules.
  \begin{equation}\label{eq:comp-rec-0}
    \infer[\text{rec-}\emptyseq]{\CompSt{ f(t, t_1, \ldots, t_n)}{ \rho }{z^*}}{
      \CompSt{ g_\emptyseq(x_1, \ldots, x_n)}{ \xi }{z}
      &
      \CompSt{ t}{ \rho }{\emptyseq}
      &
      (\CompSt{t_i}{\rho}{v_i})_{i = 1, \ldots, n}
    }
  \end{equation}
  where $\xi(x_i) = v_i$ for $i = 1, \ldots, n$.
  \begin{equation}
  \label{eq:comp-rec-succ}
    \infer[\text{rec-}b]{\CompSt{ f(t, t_1, \ldots, t_n)}{ \rho }{z^*}}{
      \CompSt{g_i(x_0, y, \Multi{x})}{\xi}{z}
      &
      \CompSt{t}{\rho}{iv_0}
      &
      \CompSt{f(x_0, \Multi{x})}{\nu}{w}
      &
      (\CompSt{t_j}{\rho}{v_j})_{j = 1, \ldots, n}
    }
  \end{equation}
  where $b = 0, 1$ and $\Multi{x} = x_1, \ldots, x_n$.
  The environment $\nu$ is defined by $\nu(x_j) = v_j$ for $j = 1, \ldots, n$ and $\nu(x_0) = v_0$, while $\xi$ is defined by $\xi(x_j) = v_j, \xi(x_0) = v_0, \xi(y) = w$.

\begin{definition}[Computation Sequence]\label{defn:comp-seq}
  A \emph{computation sequence} $\sigma$ is a sequence $\sigma_1, \ldots \sigma_L$, where each $\sigma_i$ is a sequence
  \begin{equation}
    [R, \CompSt{t_i}{\rho_i}{v_i}, n_{1i}, \ldots ,n_{{l_i}i}]
  \end{equation}
  which satisfies $n_{ji} < i, j = 1, \ldots, l_i$.
  Each inference
  \begin{equation}
    \infer[(\sigma_i)_1]{(\sigma_i)_2}{
      (\sigma_{n_{1i}})_2 &
      \cdots
      &
    (\sigma_{n_{{l_i}i}})_2
    }
  \end{equation}
  must be a valid computation rule. Here, $(a)_i$ is a projection of $[a_1, \ldots, a_n]$ to $a_i$.
  The computation judgments that are not used as assumptions of some inference rule are referred to as \emph{conclusions} of $\sigma$.
  If $\CompSt{t}{\rho}{v}$ is the only conclusion of $\sigma$, it is written as $\sigma \vdash \CompSt{ t}{ \rho }{v}$; however, if $\sigma$ has multiple conclusions $\Multi{\alpha}$, it is written as $\sigma \vdash \Multi{\alpha}$.
  If $\sigma \vdash \CompSt{t}{\rho}{v}, \Multi{\alpha}$, $\sigma$ is often considered to be a computation of $\CompSt{ t}{ \rho }{v}$.
  Although a computation sequence $\sigma$ is a sequence, $\sigma$ is often considered to be a DAG, of which the conclusions form the lowest elements.

  If there is a computation sequence $\sigma$ with conclusions $\CompSt{t}{\rho}{v}, \Multi{\alpha}$ such that $\STEPS{\sigma} \leq b$, we write $\vdash_b \CompSt{t}{\rho}{v}, \Multi{\alpha}$.

  For any sequence of computational judgments $\Multi{\alpha} = \CompSt{t_1}{\rho_1}{v_1}, \ldots, \CompSt{t_n}{\rho_n}{v_n}$, $T(\Multi{\alpha}) = \max \{\SIZE{t_1}, \ldots, \SIZE{t_n}\}$.
  For a computation $\sigma$, $M(\sigma)$ is defined as the maximal size of the main terms of computational judgments in $\sigma$, and $T(\sigma) = T(\Multi{\alpha})$ if $\Multi{\alpha}$ are conclusions of $\sigma$.
  For computational judgments $\alpha$ above, $B(\alpha)$, $L(\alpha)$, and $S(\alpha)$ are defined by $\max^n_{i = 1} B(\rho_i), \max^n_{i = 1} L(\rho_i)$ and $\max^n_{i = 1} S(\rho_i)$ respectively.
  For a computation $\sigma$ with the conclusion $\alpha$, $B(\sigma) = B(\alpha)$, $L(\sigma) = L(\alpha)$ and $S(\sigma) = S(\alpha)$.
  \end{definition}

We would like to show that $\vdash_{|B|} \CompSt{t}{\rho}{v}, \Multi{\alpha}$ is definable in $S^1_2$.
The obstacle to do this is that, in the above definition, only the number of nodes of $\sigma$, and not the number of primitive symbols, is bounded.
Thus, it is required to bound polynomially $\SIZE{\sigma}$ by $\STEPS{\sigma}$.
This task is carried out in Section \ref{sec:estimate}.

\section{Estimating the size of a computation}\label{sec:estimate}

This section proves the polynomial upper bound of the size of a computation with respect to the number of nodes of the computation together with the size of its conclusions.
($S^1_2$) means that a statement is provable in the theory $S^1_2$, and ($S^2_2$) means that a statement is provable in the $S^2_2$.

\begin{lemma}[$S^1_2$]\label{lem:f}
  Let $\sigma$ be a computation of $\CompSt{t_1}{\rho_1}{v_1}, \ldots, \CompSt{t_m}{\rho_m}{v_m}$.
  Then, $\Base(t_1, \ldots, t_m)$ contains all function symbols that appear in the main terms of $\sigma$.
  If $f(x_1, \ldots, x_{n})$ is a main term that appears in $\sigma$, $\SIZE{f} \leq T(\sigma)$ and
  \begin{align}
    \SIZE{f(x_1, \ldots, x_{n})} &\leq T(\sigma) + 2 + \sum_{i=1}^n \SIZE{x_i}\\
    &\leq p_M(T(\sigma)) \label{eq:p_M}
  \end{align}
    for a polynomial $p_M$.
\end{lemma}

\begin{proof}
  The first half of the lemma is proven by induction on $\sigma$.
  Because $f$ is contained in $\Base(t_1, \ldots, t_m)$, $\SIZE{f} \leq T(\sigma)$.
  $\SIZE{x_i}$ is polynomially increased by $\SIZE{i}$, as $x_i$ is a compound symbol constructed from the symbol $x$ and $i$.
  Because $i \leq n \leq \ar(f) \leq \SIZE{f} \leq T(\sigma)$, $\SIZE{i} \leq T(\sigma)$.
  Thus, there is a polynomial $p_M$ that satisfies \eqref{eq:p_M}.
\end{proof}

\begin{lemma}[$S^1_2$]\label{lem:v-num}
  Let $v$ is a g-numeral that appears as a value in $\sigma$.
  Then, $\STEPS{v} \leq \max \{ B(\sigma),  T(\sigma) \} + \STEPS \sigma$.
\end{lemma}

To prove this lemma, we define a wighted directed graph $G_\sigma$ and prove related lemmas.
\begin{definition}\label{defn:G_sigma}
  In a computational judgement $\CompSt{t}{\rho}{v}$, we call $v$ the righthand and $\Bracket{t}{\rho}$ the lefthand.
  The nodes of the weighted directed graph $G_\sigma$ consist of all righthands and lefthands of computational judgements of $\sigma$.
  For each node $\eta$, we define an integer $N(\eta)$ as $\STEPS{v}$ if $\eta$ is a value $v$, and as $\max(N(t), B(\rho))$, where $N(t)$ is defined as the maximal $\STEPS{v}$ of numerals $v$ which are contained in $t$, if $\eta$ is the lefthand $\Bracket{t}{\rho}$.
  For each inference of $\sigma$, edges of $G_\sigma$ are defined as follows.

  \begin{equation}
    \infer[*]{\CompSt{t}{\rho}{*}}{
    }
  \end{equation}
  In this case, we connect a edge from $t$ to $*$.  The weight is 0.

  \begin{equation}
    \infer[\Env]{\CompSt{x}{\rho[x \mapsto v]}{v^*}}{
    }
  \end{equation}
  In this case, we connect a edge from $\Bracket{x}{\rho[x \mapsto v]}$ to the
  the righthand side $v^*$.
  The weights of the edge are 0.

  \begin{equation}
    \infer[v]{\CompSt{v}{\rho}{v^*}}{
    }
  \end{equation}
  In this case, we connect a edge from $v$ to $v^*$.  The weight is 0.

  \begin{equation}
    \infer[b]{\CompSt{bt}{\rho}{bv^*}}{
      \CompSt{t}{\rho}{v}}
  \end{equation}
  In this case, we connect the lefthand of the lower judgement to the lefthand of the upper judgement and the righthand of the upper judgement to the righthand of the lower judgement.
  The weight is 0 for the edge which connects the lefthand and 1 for the edge which connects the righthand.

  \begin{equation}
    \infer[\emptyseq^n]{\CompSt{\emptyseq^n(t_1, \ldots, t_n)}{\rho}{\emptyseq}}{
      (\CompSt{t_i}{\rho}{v_i})_{i=1, \ldots, n}
    }
  \end{equation}
  In this case, we connect the lefthand of the conclusion to all lefthands of the premises.
  The weights are all 0.  $\emptyseq$ is connected from its lefthand side of the judgement.  The weight is 0.

  \begin{equation}
    \infer[\proj^i_n]{\CompSt{\proj_n^i(t_1, \ldots, t_n)}{\rho}{v^*_i}}{
      (\CompSt{t_j}{\rho}{v_j})_{j = 1, \ldots, n}
    }
  \end{equation}
  In this case, we connect the lefthand of the conclusion to all lefthands of the premises, and $v_i$ in the premise to $v^*_i$ in the conclusion.
  The weights are all 0.

  \begin{equation}
    \infer[\textup{comp}]{\CompSt{f(t_1, \ldots, t_n)}{\rho}{z}}{
      \CompSt{g(\Multi{y})}{\mu}{z}
      &
      \CompSt{h_1(\Multi{x})}{\nu}{w_1}
       \cdots
      \CompSt{h_m(\Multi{x})}{\nu}{w_m} &
      (\CompSt{t_i}{\rho}{v_i})_{i=1, \ldots, n}
    }
  \end{equation}
  In this case, we connect the lefthand of the conclusion to all $\Bracket{t_i}{\rho}, i = 1, \ldots, m$.
  The weights are 0.
  Further, all $v_i, i = 1, \ldots, n$ in the premises are connected to $\Bracket{h_j(\Multi{x})}{\nu}, j = 1, \ldots, m$.
  $z$ in the premises is connected to $z$ in the conclusion.
  The weights are all 0.

  \begin{equation}
    \infer[\text{rec-}\emptyseq]{\CompSt{ f(t, t_1, \ldots, t_n)}{ \rho }{z}}{
      \CompSt{ g_\emptyseq(\Multi{x})}{ \nu }{z}
      &
      \{ \CompSt{ t}{ \rho }{\emptyseq} \}
      &
      (\CompSt{t_i}{\rho}{v_i})_{i=1, \ldots, n}
    }
  \end{equation}
  In this case, $\Bracket{t}{\rho}$ and $\Bracket{t_i}{\rho}, i = 1, \ldots, n$ in the premises are connected from $\Bracket{f(t, t_1, \ldots, t_n)}{\rho}$.
  $v_i, i = 1, \ldots, n$ is connected to $\Bracket{g_\emptyseq(\Multi{x})}{\nu}$.
  $z$ in the premise is connected to $z$ in the conclusion.
  The weights are all 0.

  \begin{equation}
    \infer[\text{rec-}b]{\CompSt{ f(t, t_1, \ldots, t_n)}{ \rho }{z}}{
      \CompSt{g_b(x_0,  y, \Multi{x})}{\xi}{z}
      &
      \{ \CompSt{t}{\rho}{bv_0} \}
      &
      \CompSt{f(y, \Multi{x})}{\nu}{w}
      &
      (\CompSt{t_i}{\rho}{v_i})_{i=1, \ldots, n}
    }
  \end{equation}
  In this case, $\Bracket{t}{\rho}$ and $\Bracket{t_i}{\rho}, i = 1, \ldots, n$ in the premises are connected from $\Bracket{f(t, t_1, \ldots, t_n)}{\rho}$.
  All $v_i, i = 1, \ldots, n$ and $bv_0$ are connected to $\Bracket{f(y, \Multi{x})}{\nu}$.
  $bv_0$, all $v_i, i = 1, \ldots, n$ and $w$ are connected to $\Bracket{g_i(x_0,  y, \Multi{x})}{\xi}$.
  Finally, $z$ in the premise is connected to $z$ in the conclusion.
  The weights are all 0.
\end{definition}

\begin{lemma}[$S^1_2$]\label{lem:backward}
  Let $\eta$ be a node of $G_\sigma$.
  Assume that from $\eta_1, \ldots, \eta_k$ the edges $e_1, \ldots, e_k$ run to $\eta$.
  If $k \geq 1$, $N(\eta) \leq w(e_i) + N(\eta_i)$ for some $i = 1, \ldots, k$ where $w(e_i)$ is the weight of $e_i$.
\end{lemma}

Let $p =  \eta_0 \stackrel{e_1}{\leftarrow} \eta_1 \stackrel{e_2}{\leftarrow} \eta_2 \cdots \stackrel{e_l}{\leftarrow} \eta_l$ be a path in $G_\sigma$.
$p$ is called a \emph{bounding path} if for each $k = 0, \ldots, l-1$, $N(\eta_k) \leq w(e_{k+1}) + N(\eta_{k+1})$ holds.

\begin{lemma}($S^1_2$)\label{lem:connect-lr}
  If a judgement $\CompSt{t}{\rho}{v}$ appears in $\sigma$, there is a bounding path from $\Bracket{t}{\rho}$ to $v$ in $G_\sigma$.
\end{lemma}

\begin{proof}
  Assume that $\sigma_i \equiv \CompSt{t}{\rho}{v}$.
  We prove the statement of the lemma by induction on $i$.
\end{proof}

\begin{lemma}($S^1_2$)\label{lem:connect-concl-l}
  If a judgement $\CompSt{t}{\rho}{v}$ appears in $\sigma$, there is a bounding path of $\Bracket{t}{\rho}$ from a lefthand of a conclusion.
\end{lemma}

\begin{proof}
  Assume that $\sigma_i \equiv \CompSt{t}{\rho}{v}$.
  We prove the statement of the lemma by induction on $i$.
\end{proof}

\begin{lemma}($S^1_2$)\label{lem:connect-concl}
  To each righthand side of an assumption of an inference in $\sigma$, there is a bounding path of $G_\sigma$ from a lefthand of a conclusion.
\end{lemma}

\begin{proof}
  By Lemmas \ref{lem:connect-lr} and \ref{lem:connect-concl-l}.
\end{proof}

\begin{lemma}($S^1_2$)
  For each node $\eta$ of $G_\sigma$, there is an acyclic path $\eta \stackrel{e_1}{\leftarrow} \eta_1 \stackrel{e_2}{\leftarrow} \eta_2 \cdots \stackrel{e_l}{\leftarrow} \eta_l$ such that
	\begin{enumerate}
		\item $\eta_l$ is a lefthand of a conclusion of $\sigma$.
		\item $N(\eta) \leq N(\eta_l)+ \sum^l_{i=1} w(e_i)$.
	\end{enumerate}
\end{lemma}

\begin{proof}
  By Lemma \ref{lem:connect-lr} and \ref{lem:connect-concl}, there is a bounding path $p$ from a lefthand of conclusion to $\eta$.
  Assume that it contains a cycle $\eta_i, \ldots, \eta_j = \eta_i$.
  Because $p$ is a bounding path, $N(\eta_i) = N(\eta_j) \leq w(e_{j+1}) + N(\eta_{j+1})$.
  Thus, $\eta_1, \ldots, \eta_i, \eta_{j+1}, \eta_l$ is also a bounding path.
  In this way, we can remove all cycles from $p$.
\end{proof}

\begin{proof}[Proof of Lemma \ref{lem:v-num}]
  For each value $v$ which appears in a judgement of $\sigma$, $N(v) \leq N(\Bracket{t_k}{\rho_k}) + \sum_{i = 1}^l w(e_i)$ holds, where $v \stackrel{e_1}{\leftarrow} \eta_1 \stackrel{e_2}{\leftarrow} \eta_2 \cdots \stackrel{e_l}{\leftarrow} \eta_l =  \Bracket{t_k}{\rho_k}$ is an acyclic bounding path from a righthand of a conclusion to $v$.
  $N(\Bracket{t_k}{\rho_k}) \leq \max(T(\sigma), B(\sigma))$ holds.
  $\sum_{i = 1}^l w(e_i)$ is bounded by the number of $b$-rules in $\sigma$.
  Thus, $N(v)$ is bounded by $\max(T(\sigma), B(\sigma)) + \STEPS{\sigma}$.
\end{proof}

\begin{lemma}[$S^1_2$]\label{lem:sigma-rho}
  Let $\sigma$ be a computation of $\CompSt{t_1}{\rho_1}{v_1}, \ldots, \CompSt{t_m}{\rho_m}{v_m}$.
  Then, for any environment $\rho$ that appears in $\sigma$, $L(\rho) \leq \max \{L(\sigma), T(\sigma)\}$, $B(\rho) \leq \max \{ B(\sigma),  T(\sigma) \} + \STEPS \sigma$ and $S(\rho) \leq p_S(B(\sigma), T(\sigma), \STEPS{\sigma})$ for some polynomial $p_S$.
\end{lemma}

\begin{proof}
  $L(\rho) \leq \max \{L(\sigma), T(\sigma)\}$ holds because $\ar(f) \leq T(\sigma)$ for $f$ that appears in $\sigma$.
  $B(\rho) \leq \max \{ B(\sigma),  T(\sigma) \} + \STEPS \sigma$ holds by Lemma \ref{lem:v-num}.
  The polynomial bound for $S(\sigma)$ is obtained from the bounds of $L(\sigma)$ and $B(\sigma)$ together with the fact that all variables in $\rho$ appear in either environments of conclusions of $\sigma$ or $x_1, \ldots, x_{\ar(f)}, f \in \Base(\Multi{\alpha})$ where $\Multi{\alpha}$ are the conclusions of $\sigma$.
\end{proof}

\begin{lemma}($S^1_2$)\label{lem:sigma-size}
  There is a polynomial $p$ such that if there is a computation $\sigma$ of $\CompSt{t_1}{\rho_1}{v_1}, \ldots, \CompSt{t_m}{\rho_m}{v_m}$ then $\SIZE{\sigma} \leq p(\SIZE{\Multi{t}}, \SIZE{\Multi{\rho}}, \STEPS{\sigma})$ holds.
\end{lemma}

\begin{proof}
  By Lemma \ref{lem:f}, \ref{lem:v-num} and \ref{lem:sigma-rho}.
\end{proof}

\begin{lemma}\label{lem:sigma} The relation $\vdash_{\BLEN{B}} \Multi{\alpha}$ on integers $B$ and judgments $\Multi{\alpha}$ can be defined using a $\Sigma^b_1$-formula.
\end{lemma}

\begin{proof}
  Immediate from Lemma \ref{lem:sigma-size}.
\end{proof}

\section{Basic properties of computations}\label{sec:basics}

In this section, the basic properties of computations are proved.
After proving technical lemmas (Lemma \ref{lem:concl+1}, \ref{lem:sigma-rem}), we prove the lemmas concerning the forms of values of computations of $\emptyseq, 0t, 1t$ and numerals (Lemma \ref{lem:succ}, \ref{lem:num}).
Lemma \ref{lem:num} is crucial for our consistency proof because it shows that the numerals are only computed to the equal numerals.
Next we prove Lemma \ref{lem:compat}, which states the values $v_1$ and $v_2$ obtained by computations of the same term $t$ are always compatible, that is, either $v_1 \APPROXEQ v_2$ or $v_2 \APPROXEQ v_1$.
This enables us to extract the most ``accurate'' value $v(t, \rho, \sigma)$ of a term $t$ from a computation $\sigma$ of $t$ under an assignment $\rho$ (Definition \ref{defn:v}).
Substitution lemmas (Lemma \ref{lem:subst-I}, \ref{lem:subst-II}) establish the relation between substitution into a term which is evaluated by a computation, and assignment in the environment in which the term is evaluated.
These lemmas enable us to extend a consistency proof to the substitution rule.
Unlike other lemmas in this section, substitution lemmas are proved in $S^2_2$.
Finally, we prove Lemmas \ref{lem:defn} and \ref{lem:defn-2} which are used to prove ``soundness'' of defining axioms in Section \ref{sec:consistency}.

\begin{lemma}[$S^1_2$]\label{lem:concl+1}
  If $\CompSt{t}{\rho}{v}$ appears as a node in the computation sequence $\sigma, \sigma \vdash \Multi{\alpha}$ (as a DAG), there is a computation sequence $\tau$ such that $\tau \vdash \CompSt{t}{\rho}{v}, \Multi{\alpha}$, $\STEPS{\tau} \leq \STEPS{\sigma} + 1$ and $M(\tau) = M(\sigma)$.
\end{lemma}

\begin{proof}
  If $\CompSt{t}{\rho}{v}$ is derived according to the inference $R$,
  another instance of $R$ is added to $\sigma$, which uses the same assumptions
  as $R$, in which case $\tau$ is obtained.
\end{proof}

\begin{lemma}[$S^1_2$]\label{lem:sigma-rem}
  If there is a computation $\sigma$ such that $\sigma \vdash \alpha, \Multi{\alpha}$, then there exists $\tau$
  such that $\tau \vdash \Multi{\alpha}$ and $\STEPS{\tau} \leq
   \STEPS{\sigma}$.
\end{lemma}

\begin{lemma}[$S^1_2$]\label{lem:succ}
  If $\CompSt{\emptyseq}{\rho}{v}$ is contained in a computation
  $\sigma$, then either $v \equiv \emptyseq$ or $v \equiv *$.
  If $ \CompSt{bt}{\rho}{v}$, where $t$ is not a numeral, is contained in $\sigma$, then either $v \equiv i v_0$ for some g-numeral $v_0$ or $v \equiv *$.
  If $v \equiv i v_0$, then $\sigma$ contains $\CompSt{t}{\rho}{v'_0}$ and $v_0' \APPROXEQ v_0$.
\end{lemma}

\begin{proof}
  By induction on $\STEPS{\sigma}$.
  The only rule that can derive $\CompSt{\emptyseq}{\rho}{v}$ is either $v$ or $*$-rule.
  Thus, $v$ is either $\emptyseq$ or $*$.
  Similarly, if $\sigma$ derives $\CompSt{bt}{\rho}{v}$ , the only rule that can derive this is $*$ or $i$-rule.
  If $v \equiv iv_0$, $\CompSt{bt}{\rho}{iv_0}$  can only be derived by $i$.
  Thus, the assumptions contain $\CompSt{t}{\rho}{v'_0}$ where $v'_0 \APPROXEQ v_0$.
\end{proof}

\begin{lemma}[$S^1_2$]\label{lem:num}
  If $\CompSt{v}{\rho}{w}$, in which $v$ is a numeral, is contained in a computation $\sigma$, then $v \APPROXEQ w$.
\end{lemma}

\begin{proof}
  The only rules that can derive $\CompSt{v}{\rho}{w}$ are $*$ and $v$-rules.
\end{proof}

\begin{lemma}[$S^1_2$]\label{lem:compat}
  Let $t$ be a term and $v, w$ are g-numerals.
  If both $\CompSt{t}{\rho}{v}$ and $\CompSt{t}{\rho}{w}$ are present in a computation, then $v \APPROXEQ w$ or $w \APPROXEQ v$.
\end{lemma}
We write $v \COMPATIBLE w$ when $v \APPROXEQ w$ or $w \APPROXEQ v$.
$\COMPATIBLE$ is reflexive and symmetric, but not transitive.

To prove the lemma, the next lemma, same as Lemma 4.4. in \cite{Beckmann2002}, is to be observed.
\begin{lemma}[$S^1_2$]\label{lem:beckmann44}
  If g-numerals $w, u, v$ have relations $w \APPROXEQ u, w \APPROXEQ v$, then $u \COMPATIBLE v$.
\end{lemma}

\begin{proof}
  By induction on $\STEPS{w} + \STEPS{u} + \STEPS{v}$.

  If $w = *$, $u = v = *$; therefore, $u \COMPATIBLE v$.
  If $w = \emptyseq$, $u, v$ are either $*$ or $\emptyseq$.
  Therefore, $u \COMPATIBLE v$.
  If $w = bw'$ for some $b = 0, 1$ and $u'$, there are two possibilities of $u$.
  \begin{enumerate}
    \item $u$ is $*$
    \item $u$ is $bu'$
  \end{enumerate}
   If $u$ is $*$, $u \COMPATIBLE v$.
   Therefore, $u \COMPATIBLE v$.
   If $u$ is $bu'$, $w' \APPROXEQ u'$.
   There are also two possibilities of $v$.
   The only non-trivial case is the case in which $v = bv'$.
   The other case is symmetric.
   By induction hypothesis, $u' \COMPATIBLE v'$.
   Therefore, $u \COMPATIBLE v$.
\end{proof}

\begin{corollary}[$S^1_2$]\label{cor:compat}
  If $u \COMPATIBLE v$ for g-numerals $u, v$ and $u \APPROXEQ u'$ and $v \APPROXEQ v'$, then $u' \COMPATIBLE v'$.
\end{corollary}

\begin{proof}
  We assume $u \APPROXEQ v$. Then, $u \APPROXEQ v \APPROXEQ v'$. Thus, $u \APPROXEQ u'$ and $u \APPROXEQ v'$. Using Lemma \ref{lem:beckmann44}, $u' \COMPATIBLE v'$.
\end{proof}

\begin{lemma}[$S^1_2$]\label{lem:compat-f}
  Let $f$ be a function symbol and $w_1, \ldots, w_m$, $v_1, \ldots, v_m$ be g-numerals such that $w_i \COMPATIBLE v_i$ for any $i = 1, \ldots, m$.
  Let $\rho(x_1) = w_1, \ldots, \rho(x_m) = w_m$ and $\nu(x_1) = v_1, \ldots, \rho(x_m) = v_m$.
  We assume that $\CompSt{f(x_1, \ldots, x_m)}{\rho}{w}$ and $\CompSt{f(x_1, \ldots, x_m)}{\nu}{v}$ are present in the same computation $\sigma$.
  Then, $w \COMPATIBLE v$.
\end{lemma}

\begin{proof}
  We assume that $\CompSt{f(x_1, \ldots, x_m)}{\rho}{w}$ is $i$-the judgment of $\sigma$, while $\CompSt{f(x_1, \ldots, x_m)}{\nu}{v}$ is the $j$-th judgment of $\sigma$.
  The lemma is proved by induction on $i + j$ and case analysis of the rules that derive these statements.
  Although it appears that we use induction on $\Pi^0_1$-formula (because $f$ and g-numerals $w_1, \ldots, w_m, v_1, \ldots, v_m$ are not bounded), we actually need to consider only those that are included in the computation $\sigma$.
  Thus, the quantifier is bounded.
  Because $*$ is compatible with any g-numerals, we assume that $w$ and $v$ are not $*$.
  This assumption makes it possible to uniquely determine the label $R$ of the rules deriving $i$- and $j$-th judgment by $f$.

  The possible $R$ are among the $b$, $\emptyseq^m$, $\proj^i_m$, $\textup{comp}$, $\textup{rec-}\emptyseq$, and $\textup{rec-}b$-rule.

  The case in which $R$ is $b$, $\emptyseq^m$ or $\proj^i_m$ is trivial.

  The case in which $R$ is $\textup{comp}$ is considered. The last rules have the following forms.
  \begin{gather}
    \infer[\text{comp}]{\langle f(x_1, \ldots, x_m), \rho \rangle \downarrow w}{
      \langle g(\Multi{y}), \rho_0 \rangle \downarrow w
      &
      \langle h_1(\Multi{x}), \rho^* \rangle \downarrow z_1
       \cdots
      \langle h_k(\Multi{x}), \rho^* \rangle \downarrow z_k &
      (\CompSt{x_i}{\rho}{w_i^*})_{i = 1, \ldots, m}
    }\\
    \infer[\text{comp}]{\langle f(x_1, \ldots, x_m), \nu \rangle \downarrow v}{
      \langle g(\Multi{y}), \nu_0 \rangle \downarrow v
      &
      \langle h_1(\Multi{x}), \nu^* \rangle \downarrow z'_1
       \cdots
      \langle h_k(\Multi{x}), \nu^* \rangle \downarrow z'_k &
      (\CompSt{x_i}{\rho}{v_i^{*}})_{i = 1, \ldots, m}
    }
  \end{gather}
  where $\rho^*(x_1) = w_1^*$ for $i = 1, \ldots, m$, $\nu^*(x_i) = v_i^*$ for $i = 1, \ldots, m$, $\rho_0(y_1) = z_1, \ldots, \rho_0(y_k) = z_k$ and $\nu_0(y_1) = z'_1, \ldots, \nu_k(y_k) = z'_k$.
  By induction hypothesis, $z_i \COMPATIBLE z'_i$ for each $i = 1, \ldots, m$.
  By induction hypothesis, $w \COMPATIBLE v$.

  Finally, the case in which the last rule is either $\textup{rec}$-$\emptyseq$ or $\textup{rec}$-$b, b = 0, 1$ is considered. Because the case for $\textup{rec}$-$\emptyseq$ is similar to $\textup{comp}$, we consider the case in which the last rule is $\textup{rec}$-$b$.
  \begin{gather}
   \infer{\langle f(x_0, \Multi{x}), \rho \rangle \downarrow w}{
      \langle g_i(y, x_0, \Multi{x}), \rho^*[y \mapsto z] \rangle \downarrow w
      &
      \langle f(x_0, \Multi{x}), \rho^* \rangle \downarrow z&
      (\CompSt{x_i}{\rho}{w_i})_{i = 0, \ldots, m}
    }\\
    \infer{\langle f(x_0, \Multi{x}), \nu \rangle \downarrow v}{
       \langle g_i(y, x_0, \Multi{x}), \nu^*[y \mapsto z'] \rangle \downarrow v
       &
       \langle f(x_0, \Multi{x}), \nu^* \rangle \downarrow z'&
       (\CompSt{x_i}{\rho}{v_i})_{i = 0, \ldots, m}
     }
  \end{gather}
  Here, $\Multi{x} = x_1, \ldots, x_m$, $\rho^*(x_1) = w_1^*, \ldots, \rho^*(x_m) = w_m^*$ while $\rho^*(x_0) = w'_0$ when $w_0^* = b w'_0$.
  Similarly, $\nu^*(x_1) = v_1^*, \ldots, \nu^*(x_m) = v_m^*$, while $\nu^*(x_0) = v^*_0$ when $v_0^* = b v'_0$.
  $w_1^*, v_1^*$ must be in the above forms, otherwise the inference is not valid.
  Because $w_0 \COMPATIBLE v_0, \ldots, w_m \COMPATIBLE v_m$, $w_0^* \COMPATIBLE v_0^*, \ldots, w_m^* \COMPATIBLE v_m^*$.
  Further, by definition of $\APPROXEQ$, $w'_0 \COMPATIBLE v'_0$.
  By induction hypothesis, $z \COMPATIBLE z'$.
  Again, by applying induction hypothesis, $w \COMPATIBLE v$.
\end{proof}

\begin{proof}[Proof of Lemma \ref{lem:compat}]
  We assume that $\CompSt{t}{\rho}{v}$ is $i$-the judgment of $\sigma$, while $\CompSt{t}{\rho}{w}$ is the $j$-th judgment of $\sigma$.
  The lemma is proved by induction on $i + j$ and case analysis of the rules that derives these statements.
  Because the case in which either $v$ or $w$ is $*$ is trivial, we can assume that the rules that derive $\CompSt{t}{\rho}{v}$ and $\CompSt{t}{\rho}{w}$ has the same label $R$.

  If $R$ is $\Env$, $t \equiv x$ for a variable $x$ and both derivations have forms
  \begin{gather}
    \infer{\CompSt{x}{\rho[x \mapsto v_0]}{v}}{
    } \\
    \infer{\CompSt{x}{\rho[x \mapsto v_0]}{w}}{
    }
  \end{gather}
  where $v_0 \APPROXEQ v, w$.
  By Lemma \ref{lem:beckmann44}, $v \COMPATIBLE w$ holds.

  Thus, we can assume that $t = f(t_1, \ldots, t_n)$.
  If $f \equiv 0, 1$, then the proof is trivial by induction hypothesis.

  Otherwise, the derivations have the following forms.
  \begin{gather}
    \infer{\CompSt{f(t_1, \ldots, t_n)}{\rho}{v}}{
      \Multi{\beta} &
      (\CompSt{t_i}{\rho}{v_i})_{i = 1, \ldots, n}
    }\\
    \infer{\CompSt{f(t_1, \ldots, t_n)}{\rho}{w}}{
      \Multi{\gamma} &
      (\CompSt{t_i}{\rho}{w_i})_{i = 1, \ldots, n}
    }
  \end{gather}
  where
  \begin{align}
    \beta &= \CompSt{f_1(y^1_1, \ldots, y^1_{m_1})}{\rho_1}{z_1}, \ldots \CompSt{f_k(y^k_1, \ldots, y^k_{m_1})}{\rho_k}{z_k}\\
    \gamma &= \CompSt{f_1(y^1_1, \ldots, y^1_{m_1})}{\nu_1}{z'_1}, \ldots \CompSt{f_k(y^k_1, \ldots, y^k_{m_1})}{\nu_k}{z'_k}.
  \end{align}
  By induction hypothesis, $v_i \COMPATIBLE w_i$ for all $i = 1, \ldots, n$.
  Using Lemma \ref{lem:compat-f} repeatedly, we obtain $z_1 \COMPATIBLE z_1'$.
  Because $z_1 \APPROXEQ v$ and $z_1' \APPROXEQ w$, $v \COMPATIBLE w$.
\end{proof}

Note that if $v_1 \COMPATIBLE v_2$, the infimum of $v_1$ and $v_2$ exists.
This fact enables the following definition.

\begin{definition}\label{defn:v}
  Let $t$ be a term of $\PV^-$, and $\sigma$ be a computation.
  Assume $\sigma$ contains computational judgments $\CompSt{t}{\rho}{v_1}, \ldots, \CompSt{t}{\rho}{v_n}$.
  By Lemma \ref{lem:compat}, $v_1, \ldots, v_n$ are compatible.
  $v(t, \rho, \sigma)$ is defined as infimum of them.
  If there is no computational judgment of $t$ with the environment $\rho$ in $\sigma$, $v(t, \rho, \sigma) = *$.
\end{definition}

\begin{lemma}\label{lem:extract-v}
  Let $\sigma$ be a computation, $t$ be a term, and $\rho$ be an environment.
  Let $v = v(t, \rho, \sigma)$.
  If $v$ is a g-numeral other than $*$, $\sigma$ contains a judgment $\CompSt{t}{\rho}{v}$.
\end{lemma}

\begin{lemma}[$S^2_2$, Substitution Lemma I]\label{lem:subst-I}
  Let $\sigma$ be a computation that contains occurrences of judgments
  \begin{equation}
    \CompSt{t_1(u_1, \ldots, u_n)}{\rho}{v_1}, \ldots, \CompSt{t_m(u_1, \ldots, u_n)}{\rho}{v_m}.
  \end{equation}
  as conclusions.
  Assume that $v(u_i, \rho, \sigma) \APPROXEQ w_i$ for $i = 1, \ldots, n$ and let $\rho' = \rho[x_1 \mapsto w_1, \ldots, x_n \mapsto w_n]$ where $x_1, \ldots, x_n$ are fresh variables.
  Then, there is a computation $\tau$ such that $\STEPS{\tau} \leq \STEPS{\sigma} + \sum_{j=1}^m \SIZE{t_j(\emptyseq, \ldots, \emptyseq)}$ and $\tau$ has conclusions
  \begin{equation}
  \CompSt{t_1}{\rho'}{v_1}, \ldots, \CompSt{t_m}{\rho'}{v_m}.
  \end{equation}
  Further, $\tau$ contains all judgments in $\sigma$ as judgments and all conclusions in $\sigma$ as conclusions.
\end{lemma}
Each $\CompSt{t_j(u_1, \ldots, u_n)}{\rho}{v_j}, j = 1, \ldots, m$ is an occurrence of a judgment but denoted as if it is a judgment by abusing notations.
Similarly, $u_1, \ldots, u_n$ in each $\CompSt{t_j(u_1, \ldots, u_n)}{\rho}{v_j}, j = 1, \ldots, m$ are occurrences of terms but denoted as if they are terms.

\begin{proof}
  Let $U$ be a fixed integer larger then $\STEPS{\sigma} + \sum_{j=1}^m \SIZE{t_j(\emptyseq, \ldots, \emptyseq)}$.
  By induction on
  \begin{equation}
  \sum \{ \SIZE{s_d(\emptyseq, \ldots, \emptyseq)} \mid \CompSt{s_d(u_1, \ldots, u_n)}{\rho}{z_d} \in A \}
  \end{equation}
  and subinduction on $\STEPS{\kappa}$, we prove the following induction hypothesis (Claim \ref{claim:subst-I}).
  \begin{claim}\label{claim:subst-I}
    Let $\kappa$ be a computation with distinguished occurrences of judgments
    \begin{equation}
      A \equiv \CompSt{s_1(u_1, \ldots, u_n)}{\rho}{z_1}, \ldots, \CompSt{s_k(u_1, \ldots, u_n)}{\rho}{z_k}
    \end{equation}
    among conclusions and satisfies
    \begin{align}
      \STEPS{\kappa} &\leq U - \sum_{d=1}^k \SIZE{s_d(\emptyseq, \ldots, \emptyseq)} \label{eq:lem1-nodes}\\
      T(\kappa) &\leq T(\sigma) \label{eq:lem1-T}\\
      B(\kappa) &\leq B(\sigma). \label{eq:lem1-B}
    \end{align}
    Further, $\kappa$ contains all judgments of $\sigma$.
    Then, there is a computation $\lambda$ that has all conclusions of $\kappa$ and $\CompSt{s_1}{\rho'}{z_1}, \ldots, \CompSt{s_k}{\rho'}{z_k}$ as conclusions.
    $\lambda$ satisfies
    \begin{align}
      \STEPS{\lambda} &\leq \STEPS{\kappa} + \sum_{d=1}^k \SIZE{s_d(\emptyseq, \ldots, \emptyseq)}
    \end{align}
    and contains all judgments in $\kappa$.
  \end{claim}
  The claim is a $\Pi^b_2$-formula because first $\kappa$ is universally quantified and $\lambda$ is existentially quantified next.
  Because $\kappa$ changes through the induction steps, quantification over $\kappa$ is necessary.
  The quantification on $\kappa$ is polynomially bounded, because of the conditions of \eqref{eq:lem1-nodes}, \eqref{eq:lem1-T}, and \eqref{eq:lem1-B}.
  The quantification of $\lambda$ is also polynomially bounded, because $T(\lambda)$ and $B(\lambda)$ are polynomially bounded by $|\kappa|$.
  Therefore, the proof can be formalized by $\Pi^b_2$-PIND.
  From the claim, our lemma is readily proven.
  Therefore, our proof can be formalized in $S^2_2$.

  We can safely assume that the last judgment is $\CompSt{s_1(u_1, \ldots, u_n)}{\rho}{w_1}$.
  We use case analysis of the last rule of $\kappa$.
  Because $\kappa$ contains all judgments of $\sigma$, $v(u_i, \rho, \kappa) \APPROXEQ w_i$ for $i = 1, \ldots, n$.

  If the last rule of $\kappa$ is either $*$ or $\Env$-rule, the proof is trivial.

  If the last rule of $\kappa$ is $v$-rule, $s_1(u_1, \ldots, u_n) \equiv b_1 \cdots b_l(u_1)$ and $u_1$ is a numeral.
  First we remove $\CompSt{s_1}{\rho}{v_1}$ from $A$ and use the induction hypothesis.
  We obtain the computation $\tau_1$.
  We add the rules
  \begin{equation}
    \infer*{\CompSt{b_1 \cdots b_l(x)}{\rho'}{b_1 \cdots b_l z_1^*}}{
        \infer[\Env]{\CompSt{x}{\rho'}{z_1^*}}{}
        }
  \end{equation}
  to $\tau_1$.
  Let this computation $\tau$.
  \begin{equation}
    \STEPS{\tau} \leq \STEPS{\sigma} + 1 + \sum_{d = 2}^k \SIZE{s_d(\emptyseq, \ldots, \emptyseq)}
  \end{equation}
  Therefore, we can prove the lemma.

  The case in which the last rule of $\kappa$ is either $0, 1$, $\emptyseq^m$, $\proj$, comp or rec-rule.
  Then, the computation rule has a form
  \begin{equation}
    \infer[\pair]{\CompSt{f(r_1(u_1, \ldots, u_n), \ldots, r_l(u_1, \ldots, u_n))}{\rho}{w_1}}{
      \Multi{\beta} &
      (\CompSt{r_q(u_1, \ldots, u_n)}{\rho}{p_q})_{q = 1, \ldots, l}
    }
  \end{equation}
  where $\Multi{\beta}$ is purely numerical, which can be empty.
  Let $\kappa_1$ be a computation obtained by making $(\CompSt{r_q(u_1, \ldots, u_n)}{\rho}{p_q})_{q = 1, \ldots, l}$ conclusions.
  This increases $\STEPS{\kappa_1}$ from $\STEPS{\kappa}$ at most $l$.
  By explicitly counting parentheses and a comma, $\sum_{q=1}^l \SIZE{r_q(\emptyseq, \ldots, \emptyseq)} + l + 2 \leq \SIZE{s_1(\emptyseq, \ldots, \emptyseq)}$.
  Because
    \begin{align}
      &\STEPS{\kappa_1}\\
      &\leq \STEPS{\kappa} + l  \\
      &\leq U - \sum_{d=1}^k \SIZE{s_d(\emptyseq, \ldots, \emptyseq)} + l \\
      &\leq U - \sum_{q=1}^l \SIZE{r_q(\emptyseq, \ldots, \emptyseq)} - l - 2 - \sum_{d=2}^k \SIZE{s_d(\emptyseq, \ldots, \emptyseq)} + l\\
      &\leq U - \sum_{q=1}^l \SIZE{r_q(\emptyseq, \ldots, \emptyseq)} - \sum_{d=2}^k \SIZE{s_d(\emptyseq, \ldots, \emptyseq)} - 2
    \end{align}
  we can apply the induction hypothesis to $\kappa_1$.
  Therefore, we obtain $\lambda_1$ that contains all conclusions of $\kappa$ plus $(\CompSt{r_q}{\rho'}{p_q})_{q=1, \ldots, l}, (\CompSt{s_d}{\rho'}{z_d})_{d = 1, \ldots, m}$ as conclusions.
  By adding one rule to $\lambda_1$, we obtain $\lambda$.
  \begin{align}
    &\STEPS{\lambda}\\
     &\leq \STEPS{\lambda_1} + 1\\
    &\leq \STEPS{\kappa_1} + \sum_{q=1}^l \SIZE{r_q(\emptyseq, \ldots, \emptyseq)} + \sum_{i = 2}^k \SIZE{s_d(\emptyseq, \ldots, \emptyseq)} + 1\\
    &\leq \STEPS{\kappa} + l + \sum_{q=1}^l \SIZE{r_q(\emptyseq, \ldots, \emptyseq)} + \sum_{d = 2}^k \SIZE{s_d(\emptyseq, \ldots, \emptyseq)} + 1\\
    &\leq \STEPS{\kappa} + \sum_{d = 1}^k \SIZE{s_d(\emptyseq, \ldots, \emptyseq)} - 1
  \end{align}

  By construction, $\lambda$ contains all judgments in $\kappa$.
\end{proof}

\begin{lemma}[$S^2_2$, Substitution Lemma II]\label{lem:subst-II}
  Let $\sigma$ be a computation with conclusions $\CompSt{t_j(x_1, \ldots, x_n)}{\rho[x_1 \mapsto w_1, \ldots, x_n \mapsto w_n]}{v_j}$ for $j = 1, \ldots, m$.
  We assume that variables $x_1, \ldots, x_n$ are not in the domain of $\rho$ and does not appear in the main terms of conclusions, except $t_1, \ldots, t_m$.
  Assume that $v(u_i, \rho, \sigma) \APPROXEQ w_i$ for $i = 1, \ldots, n$.
  Then, there is a computation $\tau$ that has all the conclusions of $\sigma$ plus $\CompSt{t_j(u_1, \ldots, u_n)}{\rho}{v_j'}$ where $v_j' \APPROXEQ v_j$ for $j = 1, \ldots, m$ as conclusions and
  \begin{equation}
    \STEPS{\tau} \leq \STEPS{\sigma} + \sum_{j = 1}^m \SIZE{t_j(\emptyseq, \ldots, \emptyseq)}.
  \end{equation}
  Further, $\tau$ contains all judgments in $\sigma$.
\end{lemma}

\begin{proof}
  Similar to Lemma \ref{lem:subst-I}, let $U$ be an integer larger than $\STEPS{\sigma} + \sum_{j=1}^m \SIZE{t_j(\emptyseq, \ldots, \emptyseq)}$.
  By induction on
  \begin{equation}
  \sum \{ \SIZE{s_d(\emptyseq, \ldots, \emptyseq)} \mid \CompSt{s_d(x_1, \ldots, x_n)}{\rho}{z_d} \in A \}
  \end{equation}
  and subinduction on $\STEPS{\kappa}$, we prove the following induction hypothesis (Claim \ref{claim:subst-II}).
  \begin{claim}\label{claim:subst-II}
    Let $\rho' = \rho[x_1 \mapsto w_1, \ldots, x_n \mapsto w_n]$.
    Let $\kappa$ be a computation with conclusions
    \begin{equation}
      A \equiv \CompSt{s_1(x_1, \ldots, x_n)}{\rho'}{z_1'}, \ldots, \CompSt{s_k(x_1, \ldots, x_n)}{\rho'}{z_k'}
    \end{equation}
    Assume that $\kappa$ contains all judgments of $\sigma$.
    We assume that variables $x_1, \ldots, x_n$ do not appear in the main terms of conclusions, except $s_1, \ldots, s_k$.
    Further, we assume that $\kappa$ satisfies
    \begin{align}
      \STEPS{\kappa} &\leq U - \sum_{d = 1}^k s_d(\emptyseq, \ldots, \emptyseq) \label{eq:lem2-nodes}\\
      T(\kappa) &\leq T(\sigma) \label{eq:lem2-T}\\
      B(\kappa) &\leq B(\sigma) \label{eq:lem2-B}.
    \end{align}
    Then, there is a computation $\lambda$ of which the conclusions are $\CompSt{s_d(u_1, \ldots, u_n)}{\rho}{z_d'}$ where $z_d' \APPROXEQ z_d$ for $d = 1, \ldots, k$ and
    \begin{equation}
      \STEPS{\lambda} \leq \STEPS{\kappa} + \sum_{d = 1}^k \SIZE{s_d(\emptyseq, \ldots, \emptyseq)}.
    \end{equation}
    Further, $\lambda$ contains all judgments in $\kappa$.
  \end{claim}
  As Claim \ref{claim:subst-I}, the claim is $\Pi^b_2$-formula, because first $\kappa$ is universally quantified and $\lambda$ is existentially quantified next.
  Because $\kappa$ changes through the induction steps, quantification over $\kappa$ is necessary.
  The quantification on $\kappa$ is polynomially bounded, because of the conditions of \eqref{eq:lem2-nodes}, \eqref{eq:lem2-T}, and \eqref{eq:lem2-B}.
  The quantification of $\lambda$ is also polynomially bounded, because $T(\lambda)$ and $B(\lambda)$ are polynomially bounded by $|\kappa|$.
  Therefore, the proof can be formalized by $\Pi^b_2$-PIND.
  Therefore, our proof can be formalized in $S^2_2$.
  From the claim, our lemma is readily proven.
  Because $\kappa$ contains all judgments of $\sigma$, $v(u_i, \rho, \kappa) \APPROXEQ w_i$.

  We can safely assume that the last judgment is $\CompSt{s_1(x_1, \ldots, x_n)}{\rho'}{z_1}$.
  We use case analysis on the last rule of $\kappa$.

  If the last rule of $\kappa$ is either $*$ or $v$-rule, the proof is trivial.

  If the last rule of $\kappa$ is $\Env$-rule for $x_i$, we replace it by $\CompSt{u_i}{\rho}{v(u_i, \rho, \kappa)}$.

  The case in which the last rule of $\kappa$ is either $\emptyseq^m$, $0, 1$, $\proj$, comp, or rec-rule.
  Then, the computation rule has the form
  \begin{equation}
    \infer{\CompSt{f(r_1, \ldots, r_l)}{\rho'}{z_1}}{
      \Multi{\beta} &
      (\CompSt{r_q}{\rho'}{p_q})_{q = 1, \ldots, l}
    }
  \end{equation}
  where $\Multi{\beta}$ is purely numerical, which can be empty.
  Let $\kappa_1$ be the computation obtained by making $(\CompSt{r_q}{\rho}{p_q})_{q = 1, \ldots, l}$ conclusions by increasing $\STEPS{\kappa}$ at most $l$.
  We add $(\CompSt{r_q}{\rho}{p_q})_{q = 1, \ldots, l}$ to $A$ while removing the occurrence of $\CompSt{s_1(x_1, \ldots, x_n)}{\rho'}{z_1}$ from $A$.
  By explicitly counting parentheses and a comma,
  \begin{equation}
    \sum_{d=1}^l \SIZE{r_q(\emptyseq, \ldots, \emptyseq)} + l + 2 \leq \SIZE{s_1(\emptyseq, \ldots, \emptyseq)}.
  \end{equation}
  Because
    \begin{align}
      &\STEPS{\kappa_1}\\
      &\leq \STEPS{\kappa} + l  \\
      &\leq U - \sum_{d=1}^k \SIZE{s_d(\emptyseq, \ldots, \emptyseq)} + l \\
      &\leq U - \sum_{q=1}^l \SIZE{r_q(\emptyseq, \ldots, \emptyseq)} - l - 2 - \sum_{d=2}^k \SIZE{s_d(\emptyseq, \ldots, \emptyseq)} + l\\
      &\leq U - \sum_{q=1}^l \SIZE{r_q(\emptyseq, \ldots, \emptyseq)} - \sum_{d=2}^k \SIZE{s_d(\emptyseq, \ldots, \emptyseq)} - 2
    \end{align}
  Thus, we can apply the induction hypothesis to $\kappa_1$.
  Therefore, we obtain $\lambda_1$ of which conclusions are $(\CompSt{r_q(u_1, \ldots, u_n)}{\rho}{p_q'})_{q=1, \ldots, l}, (\CompSt{s_d(u_1, \ldots, u_n)}{\rho}{z_d'})_{d = 2, \ldots, k}$, where $p_q' \APPROXEQ p_q$ and $z_d' \APPROXEQ z_d$.
  By adding one rule to $\lambda_1$, we obtain $\lambda$
  \begin{equation}
    \infer{\CompSt{s_1(u_1, \ldots, u_n)}{\rho}{z_1.}}{
      \Multi{\beta} &
      (\CompSt{r_q(u_1, \ldots, u_n)}{\rho}{p_q'})_{q = 1, \ldots, l}
    }
  \end{equation}
  $\STEPS{\lambda}$ is bounded by
  \begin{align}
    &\STEPS{\lambda_1} + 1\\
    &\leq \STEPS{\kappa_1} + \sum_{q=1}^l \SIZE{r_q(\emptyseq, \ldots, \emptyseq)} + \sum_{d = 2}^k \SIZE{s_d(\emptyseq, \ldots, \emptyseq)} + 1\\
    &\leq \STEPS{\kappa} + l + \sum_{q=1}^l \SIZE{r_q(\emptyseq, \ldots, \emptyseq)} + \sum_{d = 2}^m \SIZE{s_d(\emptyseq, \ldots, \emptyseq)} + 1\\
    &\leq \STEPS{\kappa} + \sum_{d = 1}^m \SIZE{s_d(\emptyseq, \ldots, \emptyseq)} - 1
  \end{align}

  By construction, $\lambda$ contains all judgments in $\kappa$.
\end{proof}

\begin{lemma}[$S^1_2$]\label{lem:defn}
  Let
  \begin{equation}\label{eq:defn}
    f(\Multi{u}) = t
  \end{equation}
  is a substitution instance of the defining axiom of a function $f$.
  If $\vdash_{|B|} \langle f(\Multi{u}), \rho \rangle \downarrow v, \Multi{\alpha}$, then
  \begin{equation}
    \vdash_{|B| + \SIZE{f(\Multi{u}) = t}} \langle t, \rho \rangle \downarrow v', \Multi{\alpha}
  \end{equation}
  such that $v' \APPROXEQ v$.
\end{lemma}

\begin{proof}
  Let $\sigma$ be a computation sequence that derives $\vdash_{|B|} \langle f(\Multi{u}), \rho
  \rangle \downarrow v, \Multi{\alpha}$.
  The lemma is proven by conducting a case analysis of the inference rule $R$ of $\CompSt{f( \Multi{u})}{\rho}{v}$ and the defining axioms of $f$.

  If $R$ is $*$-rule, the proof is obvious.
  Therefore, we assume that $R$ is not a $*$-rule.
  Then, $R$ is determined by the defining axiom of $f$.

  For the case that $f \equiv \emptyseq$ or $f \equiv 0, 1$, the
  defining axioms do not exist. Thus, the lemma vacuously holds.

  For the case in which $f \equiv \emptyseq^n$, $R$ has the form
  \begin{equation}
    \infer{\CompSt{ \emptyseq^n(u_1, \ldots, u_n)}{ \rho }{\emptyseq}.}{
      (\CompSt{u_i}{\rho}{w_i})_{i=1, \ldots, n}
    }
  \end{equation}
  Then,
  \begin{equation}
    \infer{\CompSt{ \emptyseq}{ \rho }{\emptyseq}.}{}
  \end{equation}
  The case is valid.

  For the case in which $f \equiv \proj^n_i$, $R$ has the form
  \begin{equation}
    \infer{\CompSt{ \proj^n_i(u_1, \ldots, u_n)}{ \rho }{w_i^*}.}{
    (\CompSt{u_j}{\rho}{w_j})_{j=1, \ldots, n}
    }
  \end{equation}
  Because $w_i \APPROXEQ w_i^*$, the lemma is proved.

  For the case in which $f$ is defined by the composition $g(h_1(\overline{x}), \ldots, h_n(\overline{x}))$, the inference $R$ of $\sigma$ that derives $\langle f(u_1, \ldots, u_n), \rho \rangle \downarrow v$, has the following form.
  \begin{equation}
    \infer{\CompSt{ f(u_1, \ldots, u_n)}{ \rho }{v}}{
      \CompSt{ g(\Multi{y})}{ \xi }{v}
      &
      \CompSt{ h_1(\Multi{x})}{ \nu }{w_1}
      & \cdots &
      \CompSt{h_m(\Multi{x})}{\nu}{w_m}&
      (\CompSt{u_i}{\rho}{z_i})_{i= 1, \ldots, n}
      }
  \end{equation}
  Because $\nu(x_i) = z_i, i = 1, \ldots, n$, using Lemma \ref{lem:subst-II} repeatedly, $\tau_1 \vdash \CompSt{g(\Multi{y})}{\xi}{v}, \CompSt{h_1(\Multi{u})}{\rho}{w_1'}, \ldots, \CompSt{h_m(\Multi{u})}{\rho}{w_m'}$ where $w_1' \APPROXEQ w_1, \ldots, w_m' \APPROXEQ x_m$ is obtained, in which
  \begin{equation}
    \STEPS{\tau_1} \leq \STEPS{\sigma} + \sum_j^m \SIZE{h_j(\Multi{\emptyseq})}.
  \end{equation}
  Again, using Lemma \ref{lem:subst-II}, $\tau \vdash \CompSt{g(\Multi{h}(\Multi{u}))}{\rho}{v'}, \Multi{\alpha}$ where $v' \APPROXEQ v$, is obtained where
  \begin{align}
    \STEPS{\tau} &\leq \STEPS{\tau_1} + \SIZE{g(\Multi{\emptyseq})}\\
    &\leq \STEPS{\sigma} + \SIZE{g(\Multi{\emptyseq})} + \sum_j^m \SIZE{h_j(\Multi{\emptyseq})}\\
    &\leq \STEPS{\sigma} + \SIZE{g(\Multi{h}(\Multi{u}))} + m\\
    &\leq \STEPS{\sigma} + \SIZE{f(\Multi{u}) = g(\Multi{h}(\Multi{u}))}.
  \end{align}
  The case is valid.

  If $f$ is defined by recursion, the inference that derives $\langle f(u_1, \ldots, u_n), \rho \rangle \downarrow v$ has the form.
    \begin{equation}\label{eq:last_0}
      \infer{\CompSt{ f(\emptyseq, u_2, \ldots, u_n)}{ \rho }{v}}{
        \CompSt{ g_\emptyseq(\Multi{x})}{ \nu }{v} &
        \CompSt{\emptyseq}{\rho}{\emptyseq} &
        (\CompSt{u_i}{\rho}{z_i})_{i = 2, \ldots, n}
      }
    \end{equation}
    or for each $b = 0, 1$,
    \begin{equation}\label{eq:b}
      \infer{\CompSt{ f(bu, \Multi{u})}{ \rho }{v}}{
        \CompSt{ g_b(y,  \Multi{x})}{ \nu[y \mapsto w] }{v}
        &
        \CompSt{ f(x_0, \Multi{x})}{ \nu }{w}
        &
        \CompSt{bu}{ \rho }{bz_0}
        &
        (\CompSt{u_i}{\rho}{z_i})_{i = 2, \ldots, n}
        }
    \end{equation}
    where $\nu(x_0) = z_0$, while $\nu(x_k) = z_k, k = 2, \ldots, n$.
    First, consider the case of \eqref{eq:last_0}.
    According to Lemma \ref{lem:subst-II} and Lemma \ref{lem:sigma-rem}, we have $\tau$ that satisfies $\tau \vdash \langle g(u_1, \ldots, u_n), \rho \rangle \downarrow v', \Multi{\alpha}$ where $v' \APPROXEQ v$ and $\STEPS{\tau} \leq \STEPS{\sigma} + \SIZE{g(\Multi{\emptyseq})}$.
    Thus, the case has been shown to be valid.
    Next, consider the case of \eqref{eq:b}.
    According to Lemma
    \ref{lem:succ}, $\CompSt{ u}{ \rho }{z'_0}, z'_0 \APPROXEQ z_0$ is contained in $\sigma$.
    According to Lemma \ref{lem:subst-II}, we have $\tau_1$ that has the conclusion $\CompSt{f(u, \Multi{u})}{\rho}{w'}$, where $w \APPROXEQ w'$.
    Because $\tau_1$ contains all the judgments of $\sigma$, $\CompSt{ g_b(y, \Multi{x})}{ \nu[y \mapsto w] }{v}$, and $\CompSt{u_i}{\rho}{z_i}, i = 2, \ldots, n$ appear in $\tau_1$.
    Using Lemma \ref{lem:subst-II} again, we obtain $\tau \vdash \CompSt{g_b(u, f(u, \Multi{u}),\Multi{u})}{\rho}{v'}$, where $v' \APPROXEQ v$.
    \begin{align}
      \STEPS{\tau} &\leq \STEPS{\tau_1} + \SIZE{g_b(\Multi{\emptyseq})}\\
      &\leq \STEPS{\sigma} + \SIZE{f(\Multi{\emptyseq})} + \SIZE{g_b(\Multi{\emptyseq})}\\
      &\leq \STEPS{\sigma} + \SIZE{g_b(u, f(u, \Multi{u}),\Multi{u})} + 1
    \end{align}
    The case is valid.
  \end{proof}

  \begin{lemma}($S^1_2$)\label{lem:defn-2}
  Let
  \begin{equation}\label{eq:defn-2}
    f(\Multi{u}) = t
  \end{equation}
  is a substitution instance of the defining axiom of a function $f$.
  If $\vdash_{|B|} \langle t, \rho \rangle \downarrow v, \Multi{\alpha}$, then
  \begin{equation}
    \vdash_{|B| + \SIZE{f(\Multi{u}) = t}} \langle f(\Multi{u}), \rho \rangle \downarrow v', \Multi{\alpha}
  \end{equation}
  such that $v' \APPROXEQ v$.
 \end{lemma}

  \begin{proof}
	Let $\sigma$ be a computation that derives $\vdash_{|B|} \langle t, \rho \rangle \downarrow v, \Multi{\alpha}$.
  	The lemma is proven by conducting a case analysis of the defining axiom of $f$.
    Because the case in which $v \equiv *$ is trivial, we assume that $v \not\equiv *$.

    In the case in which \eqref{eq:defn-2} has the form
    \begin{equation}
      \emptyseq^n(u_1, \ldots, u_m) = \emptyseq,
    \end{equation}
    By the computation rule
    \begin{equation}
      \infer{\CompSt{\emptyseq^n(u_1, \ldots, u_m)}{\rho}{\emptyseq}}{
  		\CompSt{u_1}{\rho}{*} &\ldots&
  		\CompSt{u_m}{\rho}{*}
  	}
    \end{equation}
    $\vdash_{\SIZE{\emptyseq^n(u_1, \ldots, u_m) = \emptyseq}} \CompSt{\emptyseq^n(u_1, \ldots,
      u_m)}{\rho}{\emptyseq}$ holds.
    By Lemma \ref{lem:succ}, if $\CompSt{\epsilon}{\rho}{v}$, then $v$ is either $*$ or $\emptyseq$.
    Therefore, the case is valid.

    In the case for which \eqref{eq:defn-2} has the form
    \begin{equation}
      \proj_i^n(u_1, \ldots, u_m) = u_i,
    \end{equation}
    the computation rule
    \begin{equation}
      \infer{\CompSt{\proj_i^n(u_1, \ldots, u_m)}{\rho}{v}.}{
        \CompSt{u_1}{\rho}{*} \cdots
        \CompSt{u_{i-1}}{\rho}{*}&
  		 \CompSt{u_i}{\rho}{v}&
       \CompSt{u_{i+1}}{\rho}{*} \cdots
        \CompSt{u_m}{\rho}{*}
      }
    \end{equation}
    applies.
    By assumption, $\vdash_{|B|} \CompSt{ t_i}{ \rho }{v}$.
    Thus, $\vdash_{|B|+m}\langle \proj_i^n(t_1, \ldots, t_m), \rho \rangle
    \downarrow v$. Therefore, the case is valid.

    The case for which $f$ is defined by the composition of $g, h_1, \ldots, h_m$ is presented next.
    Let $\sigma$ be a computation of $\CompSt{g(h_1(\Multi{u}), \ldots, h_m(\Multi{u})))}{\rho}{v}$.  Then, $\sigma$ has the following form.
    \begin{equation}
      \infer{\CompSt{g(h_1(\Multi{u}), \ldots, h_m(\Multi{u}))}{\rho}{v}}{
        \Multi{\beta} &
        \CompSt{h_1(\Multi{u})}{\rho}{w_1}
        & \ldots &
        \CompSt{h_m(\Multi{u})}{\rho}{w_m}
      }
    \end{equation}
    where $\Multi\beta$ is purely numerical.
    Let $v_1 = v(u_1, \rho, \sigma), \ldots, v_n = v(u_n, \rho, \sigma)$ and $\nu(x_1) = v_1, \ldots, \nu(x_n) = v_n$.
    By repeatedly applying Lemma \ref{lem:subst-I}, we obtain a computation $\tau_1$ that has the conclusion $\CompSt{g(h_1(\Multi{x}), \ldots, h_m(\Multi{x}))}{\nu}{v}$.
    Because $v$ is not $*$, $\tau_1$ contains the inference
    \begin{equation}
      \infer{\CompSt{g(h_1(\Multi{x}), \ldots, h_m(\Multi{x}))}{\nu}{v}}{
        \Multi{\gamma} &
        \CompSt{h_1(\Multi{x})}{\nu}{w_1}
        & \ldots &
        \CompSt{h_m(\Multi{x})}{\nu}{w_m}
      }
    \end{equation}
    where $\STEPS{\tau_1} \leq \STEPS{\sigma} + \sum_{j = 1}^m \SIZE{(h_j(\Multi\emptyseq)}$.
    By applying Lemma \ref{lem:subst-I}, we obtain $\delta_1$ that contains the judgment $\CompSt{g(y_1, \ldots, y_m)}{\xi}{v}$, where $\xi(y_1) = w_1, \ldots, \xi(y_m) = w_m$, and satisfies $\STEPS{\delta_1} \leq \STEPS{\tau_1} + \SIZE{g(\Multi\emptyseq)}$.
    $\delta_1$ contains $\CompSt{h_1(\Multi{x})}{\nu}{w_1}, \ldots, \CompSt{h_m(\Multi{x})}{\nu}{w_m}$.
    $\delta_1$ also contains $\CompSt{u_i}{\rho}{v_i}$ for $i = 1, \ldots, m$ unless $v_i$ is $*$.
    If $v_i$ is $*$, we add $\CompSt{u_i}{\rho}{*}$ to $\delta_1$, increasing $\STEPS{\delta_1}$ by at most $n$.
    Then, we obtain $\delta_1$ which satisfies $\STEPS{\delta_1} \leq \STEPS{\sigma} + \sum_{j=1}^m \SIZE{h_j(\Multi\emptyseq)} + \SIZE{g(\Multi\emptyseq)} + n$.
    Using these judgements, we can assemble an inference of the judgement $\CompSt{f(x_1, \ldots, x_n)}{\rho}{v}$.
    \begin{equation}
      \infer{\CompSt{f(u_1, \ldots, u_n)}{\rho}{v}}{
        \CompSt{g(y_1, \ldots, y_m)}{\xi}{v} &
        (\CompSt{h_j(\Multi{x})}{\nu}{w_j})_{j = 1, \ldots, m} &
        (\CompSt{u_i}{\rho}{v_i})_{i = 1, \ldots, n}
      }
    \end{equation}
    Let $\tau$ be the computation that is created in this way.
    \begin{align}
    \STEPS{\tau} &\leq \STEPS{\sigma} + \sum_{j=1}^m \SIZE{h_j(\Multi\emptyseq)} + \SIZE{g(\Multi\emptyseq)} + n + 1 \\
    &\leq \STEPS{\sigma} + \SIZE{g(h_1(\Multi\emptyseq), \ldots, h_m(\Multi\emptyseq))} + m + n + 1 \\
    &\leq \STEPS{\sigma} + \SIZE{f(\Multi{u}) = t}
    \end{align}
    because
    \begin{align}
      \SIZE{g(h_1(\Multi\emptyseq), \ldots, h_m(\Multi\emptyseq))} &\leq \SIZE{t}\\
      m \leq \SIZE{[\Fun, \mathrm{comp}, g, h_1, \ldots, h_m]} &= \SIZE{f}\\
      n + 1 &\leq \SIZE{(u_1, \ldots, u_n)}.
    \end{align}

    The case for which $f$ is defined by recursion using $g_\emptyseq, g_0, g_1$ is presented next.
    For the case of $g_\emptyseq$, the proof is similar to that of the case of the composition.
    Consider the case in which the defining equation is $f(bu_0, \Multi{u}) = g_b(u_0, f(u_0, \Multi{u}), \Multi{u})$.
    Then, there exists a derivation $\sigma$ with the value of $g_b(u_0, f(u_0, \Multi{u}), \Multi{u})$.
    Let $w_0 = v(u_0, \rho, \sigma), w_i = v(u_i, \rho, \sigma), i = 2, \ldots, n$ and $v_0 = v(f(u_0, \Multi{u}), \rho, \sigma)$.
    The environment $\xi$ is defined by $\xi(x_0) = w_0, \xi(x_2) = w_2, \ldots, \xi(x_n) = w_n$ and $\xi(y) = v_0$.
    By Lemma \ref{lem:subst-I}, a computation $\tau$ with the conclusion $\CompSt{g_b(x_0, y, \Multi{x})}{\xi}{v}$ is obtained.
    We can assume that $\tau$ contains $\CompSt{f(u_0, \Multi{u})}{\rho}{v_0}$ as a conclusion by increasing $\STEPS\tau$ by one.
	 The environment $\nu$ is defined by $\nu(x_0) = w_0, \nu(x_2) = w_2, \ldots, \nu(x_n) = w_n$.
    By Lemma \ref{lem:subst-I}, a computation $\mu$ with the conclusion $\CompSt{f(x_0, \Multi{x})}{\nu}{v_0}$ is obtained.
    $\mu$ still contains $\CompSt{g_i(x_0, y, \Multi{x})}{\xi}{v}$.
    Using these judgments, we can assemble a computation $\delta$ of $\CompSt{f(bu_0, \Multi{u})}{\rho}{v}$
    \begin{equation}
      \infer{\CompSt{f(bu_0, \Multi{u})}{\rho}{v}}{
        \CompSt{g(x_0, y, \Multi{x})}{\xi}{v} &
        \infer{\CompSt{b u_0}{\rho}{b v_0}}{\CompSt{u_0}{\rho}{v_0}} &
        \CompSt{f(x_0, \Multi{x})}{\nu}{w_0} &
        (\CompSt{u}{\rho}{w_i})_{i = 1, \ldots, n}
      }
    \end{equation}
    by adding at most $n + 2$ $*$-rules to derive assumptions.
    By summing up,
    \begin{align}
     \STEPS\delta &\leq \STEPS\mu + n + 2 \\
     &\leq \STEPS\tau + n + 2 + \SIZE{f(\Multi\emptyseq)} + 1\\
     &\leq \STEPS\sigma + n + 3 + \SIZE{f(\Multi\emptyseq)} + \SIZE{g_i(\Multi\emptyseq)}\\
		&\leq \STEPS\sigma + \SIZE{f(bu_0, \Multi{u}) = g_b(u_0, f(u_0, \Multi{u}), \Multi{u})}
	 \end{align}
   because
   \begin{align}
     n + 3 &\leq \SIZE{f(bu_0, \Multi{u})}\\
     \SIZE{f(\Multi\emptyseq)} + \SIZE{g_b(\Multi\emptyseq)} &\leq \SIZE{g_b(u_0, f(u_0, \Multi{u}), \Multi{u})} + 1.
   \end{align}
  \end{proof}

\section{Consistency proof}
\label{sec:consistency}

This section proves the consistency of $\PV^-$ inside $S^2_2$.
To this end, we first prove a type of soundness $\PV^-$ by the notion of computation.
We prove that, whenever an equation $t = u$ is proved, for each computation $\sigma$ of $t$ with an environment $\rho$ whose value is $v$, there is a computation $\tau$ of $u$ with the environment $\rho$ whose value is $v^\prime, v^\prime \APPROXEQ v$.
Further, the proof is carried out in $S^2_2$.
The soundness of our semantics implies consistency because the value of $1$ is never $0$ by Lemma \ref{lem:num}; therefore, it is impossible to derive $0 = 1$.
The use of $S^2_2$, and not $S^1_2$, is essential because we need to quantify over a computation $\sigma$ and an environment $\rho$ in the proof of soundness.
This introduces two alternate quantifiers in the induction hypothesis.

\begin{proposition}[$S^2_2$] \label{prop:sound}
  Let $\pi$ be a tree-like $\PV^-$-proof that derives $t=u$.
  Then, for any environment $\rho$ for the free variables of $t$ and $u$ and computation $\sigma \vdash \CompSt{t}{\rho}{v}$, there is a computation $\tau \vdash \CompSt{u}{\rho}{v'}$ such that $v' \APPROXEQ v$,  $\STEPS{\tau} \leq \STEPS{\sigma} + \SIZE{\pi}$.
\end{proposition}

\begin{proof}
  We prove the following claim using induction on a tree-like $\PV^-$-proof $\chi$.
\begin{claim}
  Let $U$ be an integer.
  Let $\chi$ be a tree-like $\PV^-$ proof that derives $r = s$.
  Then, for any
  \begin{itemize}
    \item environment $\rho$ for free variables of $r$ and $s$ such that $B(\rho) \leq \lfloor \frac{1}{2} (U - \SIZE{\chi})^2 \rfloor$ and $L(\rho) \leq U - \SIZE{\chi}$,
    \item computational judgements $\Multi\alpha \equiv \alpha_1, \ldots, \alpha_l$ such that $M(\alpha) \leq U - \SIZE{\chi}$, $B(\alpha) \leq \lfloor \frac{1}{2} (U - \SIZE{\chi})^2 \rfloor$ and $L(\alpha) \leq U - \SIZE{\chi}$,
    \item computation $\sigma \vdash \CompSt{r}{\rho}{v}, \Multi{\alpha}$ such that $\STEPS{\sigma} \leq U - \SIZE{\chi}$,
  \end{itemize}
  there is a computation $\tau \vdash \CompSt{s}{\rho}{v'}, \Multi{\alpha}$ such that $v' \APPROXEQ v$,  $\STEPS{\tau} \leq \STEPS{\sigma} + \SIZE{\chi}$.
\end{claim}
  From the claim, the proposition is immediate by letting $U$ to be sufficiently large.
  The claim is proven by induction on $\chi$.
  Because the induction hypothesis can be written by a formula with bounded universal quantifiers and one bounded existential quantifier inside, the induction hypothesis can be written by a $\Pi^b_2$-formula with two free variable $\chi$ and $U$.
  Therefore, the claim can be proven in $S^2_2$.
  Use of $\Pi^b_2$-PIND is essential because quantification over computations is necessary to interpret the transitivity rule and quantification over environments is necessary to interpret substitution.
  The proof uses case analysis of the last rule of $\chi$.
  Because for $\SIZE{\chi} > U$ the claim vacuously holds, we can assume that $\SIZE{\chi} \leq U$.

  The case for which the conclusion of $\chi$ is a defining axiom is proven using Lemma \ref{lem:defn} and \ref{lem:defn-2}.

  The case for which the axiom is the reflexive axiom is trivial.

  The case for which the last inference of $\chi$ is a symmetry rule is trivial.

  The case for which the last inference of $\chi$ is a transitivity rule is considered next.
  \begin{equation}
    \infer{t = w}{
      \infer*[\chi_1]{t = u}{}
      &
      \infer*[\chi_2]{u = w}{}
    }
  \end{equation}
	Let $\sigma$ be a computation a conclusion of which has a $\CompSt{t}{\rho}{v}$.
	By induction hypothesis, there is a computation $\tau$ of $\CompSt{u}{\rho}{v'}$ such that $\STEPS{\tau} \leq \STEPS{\sigma} + \SIZE{\chi_1}$ and $v' \APPROXEQ v$.
  Because $\STEPS{\tau} \leq U - \SIZE{\chi} + \SIZE{\chi_1} \leq U - \SIZE{\chi_2}$, by induction hypothesis on $\chi_2$, there is a computation $\delta$ of $\CompSt{w}{\rho}{v''}$ such that $\STEPS{\delta} \leq \STEPS{\sigma} + \SIZE{\chi_1} + \SIZE{\chi_2}$ and $v'' \APPROXEQ v$.
  Thus, the claim holds.

  The case in which the last inference of $\chi$ is
  \begin{equation}
    \infer{f(u_1, \ldots, u_n) = f(s_1, \ldots, s_n),}{
      \infer*[\chi_1]{u_1 = s_1}{}
      &
      \cdots
      &
      \infer*[\chi_n]{u_n = s_n}{}
    }
  \end{equation}
  is considered.
  Let $\sigma$ be a computation of $\CompSt{f(u_1, \ldots, u_n)}{\rho}{v}$.
  Let $w_1 = v(u_1, \rho, \sigma), \ldots, w_n = v(u_n, \rho, \sigma)$.
  By Lemma \ref{lem:extract-v}, increasing $\STEPS{\sigma}$ by $n$, we obtain a computation $\sigma_0$ such that $\CompSt{u_1}{\rho}{w_1}, \ldots, \CompSt{u_n}{\rho}{w_n}$ are contained in $\sigma$ as conclusions.
  The $\sigma_0$ satisfies induction hypothesis on $\chi_1$, because
  \begin{align}
    \SIZE{f(u_1, \ldots, u_n)} &\leq \SIZE{\chi} - \SIZE{\chi_1}\\
    &\leq U - \SIZE{\chi_1}\label{eq:sideM}\\
    \STEPS{\sigma} + n &\leq U - \SIZE{\chi} + \SIZE{f(u_1, \ldots, u_n)}\\
    &\leq U - \SIZE{\chi_1}
  \end{align}
  and $\SIZE{u_i} \leq U - \SIZE{\chi_1}$ for $i = 1, \ldots, n$ by the similar reason as \eqref{eq:sideM}.
  Therefore, we can transform $\sigma$ to $\sigma_1$ that has the same conclusions to $\sigma$ except one of $\CompSt{u_1}{\rho}{w_1}$, which is replaced to $\CompSt{s_1}{\rho}{w'_1}$ where $w_1' \APPROXEQ w_1$.
  This increases $\STEPS{\sigma_1}$ by $\SIZE{\chi_1}$.
  Assume that we construct a computation $\sigma_j$ that has the same conclusions to $\sigma$, except $\CompSt{u_i}{\rho}{w_i}, i = 1, \ldots, j$, which is replaced to $\CompSt{s_i}{\rho}{w'_i}$, where $w_i' \APPROXEQ w_i$ and $\STEPS{\sigma_j} \leq \STEPS{\sigma} + n + \sum_{i = 1}^j \SIZE{\chi_i}$.
  Then $\SIZE{f(u_1, \ldots, u_n)} \leq \SIZE{\chi} - \SIZE{\chi_{j+1}} \leq U - \SIZE{\chi_{j+1}}$ and $\STEPS{\sigma_j} \leq U - \SIZE{\chi} + n + \sum_{i = 1}^j \SIZE{\chi_i} \leq U - \SIZE{\chi_{j+1}}$ hold.
  Further, $\SIZE{u_i} \leq U - \SIZE{\chi_{j+1}}$ for $i = 1, \ldots, n$ holds.
  Therefore, we can apply the induction hypothesis on $\chi_{j+1}$ to $\sigma_j$ and obtain $\sigma_{j+1}$ which has the same conclusions to $\sigma$ except $\CompSt{u_i}{\rho}{w_i}, i = 1, \ldots, j+1$, which is replaced to $\CompSt{s_i}{\rho}{w'_i}$ where $w_i' \APPROXEQ w_i$ and $\STEPS{\sigma_j} \leq \STEPS{\sigma} + n + \sum_{i = 1}^{j+1} \SIZE{\chi_i}$.
  Finally, we obtain a computation $\sigma_n$ that has the same conclusions to $\sigma$, except $\CompSt{u_i}{\rho}{w_i}, i = 1, \ldots, n$, which is replaced to $\CompSt{s_i}{\rho}{w'_i}$, where $w_i' \APPROXEQ w_i$ and $\STEPS{\sigma_j} \leq \STEPS{\sigma} + n + \sum_{i = 1}^n \SIZE{\chi_i}$.
  Let $\rho' = \rho[y_1 \mapsto w_1, \ldots, y_n \mapsto w_n]$.
  Because $\sigma_n$ has the conclusion $\CompSt{f(u_1, \ldots, u_n)}{\rho}{v}$, by Lemma \ref{lem:subst-I} we obtain a computation $\tau_1$ of $\CompSt{f(y_1, \ldots, y_n)}{\rho'}{v}$.
  $\tau_1$ contains computation judgements $\CompSt{s_1}{\rho}{w'_1}, \ldots, \CompSt{s_n}{\rho}{w'_n}$ and satisfies $\STEPS{\tau_1} \leq \STEPS{\sigma_n} + \SIZE{f(\emptyseq, \ldots, \emptyseq)}$.
  By Lemma \ref{lem:subst-II}, we obtain a computation $\tau$ of $\CompSt{f(s_1, \ldots, s_n)}{\rho}{v}$.
  \begin{align}
    &\STEPS{\tau}\\
    &\leq \STEPS{\sigma} + n + \sum_{i = 1}^n \SIZE{\chi_i} + 2 \SIZE{f(\emptyseq, \ldots, \emptyseq)}\\
    &\leq \STEPS{\sigma} + \sum_{i = 1}^n \SIZE{\chi_i} + \SIZE{f(u_1, \ldots, u_n) = f(s_1, \ldots, s_n)} \\
    &\leq \STEPS{\sigma} + \SIZE{\chi}.
  \end{align}
  Therefore, the claim holds.

  Finally, we consider the substitution rule.
  \begin{equation}
    \infer{r_0(q) = s_0(q)}{
      \infer*[\chi_1]{r_0(x) = s_0(x)}{}
    }
  \end{equation}
  Let $\sigma$ be a computation of $\CompSt{r_0(q)}{\rho}{v}$ that satisfies the conditions of the proposition.
  Let $w = v(q, \rho, \sigma)$ and $\rho' = \rho[x \mapsto w]$.
  $L(\rho') \leq L(\rho) + 1 \leq U - \SIZE{\chi} + 1 \leq U - \SIZE{\chi_1}$.
  \begin{align}
    B(\rho') &\leq \max (B(\rho), \STEPS{w})\\
    &\leq \max(B(\rho), \max(B(\sigma), M(\sigma)) + \STEPS{\sigma})\\
    &\leq \max(\lfloor \frac{1}{2} (U - \SIZE{\chi})^2 \rfloor, U - \SIZE{\chi}) + U - \SIZE{\chi}\\
    &\leq \lfloor \frac{1}{2} (U - \SIZE{\chi_1})^2 \rfloor
  \end{align}
  By increasing $\STEPS{\sigma}$ by 1, we can assume that $\sigma$ contains $\CompSt{q}{\rho}{w}$ as a conclusion.
  By Lemma \ref{lem:subst-I}, there is a computation $\sigma_1$ that derives $\CompSt{r_0(x)}{\rho'}{v}$ such that $\STEPS{\sigma_1} \leq \STEPS{\sigma} + 1 + \SIZE{r_0(\emptyseq)}$.
  It is easy to see that $\sigma_1$ satisfies assumptions of induction hypothesis for $\chi_1$.
  Therefore, there is a computation $\tau_1$ of $\CompSt{s_0(x)}{\rho'}{v'}$ such that $v' \APPROXEQ v$ and $\STEPS{\tau_1} \leq \STEPS{\sigma_1} + \SIZE{\chi_1}$.
  Finally, because the conclusion $\CompSt{q}{\rho}{w}$ is preserved by all operations above, $v(q, \rho, \tau_1) \APPROXEQ w$ holds.
  By Lemma \ref{lem:subst-II}, there is a computation $\tau$ of $\CompSt{s_0(q)}{\rho}{v'}$ such that $\STEPS{\tau} \leq \STEPS{\tau_1} + \SIZE{s_0(\emptyseq)}$ and $v' \APPROXEQ v$.
  \begin{align}
    \STEPS{\tau} &\leq \STEPS{\tau_1} + \SIZE{s_0(\emptyseq)}\\
    &\leq \STEPS{\sigma_1} + \SIZE{\chi_1} + \SIZE{s_0(\emptyseq)}\\
    &\leq \STEPS{\sigma} + 1 + \SIZE{r_0(\emptyseq)} + \SIZE{\chi_1} + \SIZE{s_0(\emptyseq)}\\
    &\leq \STEPS{\sigma} + \SIZE{\chi}
  \end{align}
  Therefore, the claim holds.
	\end{proof}

  \begin{theorem}\label{cor:consis}
    $S^{2}_2$ proves $\PV^- \not\vdash 0\emptyseq = 1\emptyseq \label{eq:consis}$
  \end{theorem}

  \begin{proof}
    Assume that there is a proof $\pi$ of $0\emptyseq = 1\emptyseq$ in $\PV^-$.
    Let $\sigma$ be a computation of $\CompSt{0\emptyseq}{\EMPTYSEQ }{0\emptyseq}$.
    By Proposition \ref{prop:sound}, there is a computation $\tau$ of $\CompSt{1\emptyseq}{\EMPTYSEQ}{0\emptyseq}$, which contradicts Lemma \ref{lem:num}.
  \end{proof}

\section{Discussion}\label{sec:discuss}

\subsection{Relation to original PV}\label{subsec:original-pv}
Cook and Urquhart's original $\PV$~\cite{cook1975feasibly} has some differences from our $\PV$.

Their $\PV$ uses lambda abstraction to create new function symbols from terms.
Because Cook and Urquhart's $\PV$ uses only lambda abstraction for first-order variables, the functions that are defined by lambda abstraction can be defined by compositions, projections, and constant functions.

Another difference is that the intended domain of Cook and Urquhart's $\PV$ is the set of natural numbers.
Natural numbers are represented by the constant $0$ and binary successors $s_0, s_1$ of which the intended meaning is $2 \cdot x$ and $2 \cdot x + 1$, respectively.
On the other hand, our formalism uses the set of binary strings as the intended domain.
Our system can interpret natural numbers by using a binary number system, using little endian (the least significant bit appears at the right most position).
Then, using our system, we can define all polynomial time functions.
However, the schema of limited recursion on notation
\begin{align}\label{eq:limited-rec}
  R[g, h, k](x, \Multi{y}) &= \Cond(x, g(\Multi{y}), Cond(t \minusdot k(x, \Multi{y}), t, k(x, \Multi{y}))) \\ \notag
  t &\equiv h(x, \Multi{y}, R[g, h, k](\divtwo{x}, \Multi{y}))
\end{align}
which appears in Cook and Urquhart's $\PV$ would not be derived by our system.
This is because to derive \eqref{eq:limited-rec}, it appears that the case analysis on $x$, which does not seem to be derived from our system, is required.

\subsection{Beckmann's counter-example}\label{subsec:Beckmann}

The proof in the previous draft \cite{Yamagata2014} allows a counter-example, which was pointed out by Arnold Beckmann \cite{Beckmann2015}.
Let $g(x)$ be the function defined by $g(\emptyseq) = \emptyseq, g(0x) = \overbrace{0 \cdots 0}^k g(x), k \geq 1$.
$h(x)$ is defined recursively by $h(\emptyseq)=\emptyseq, h(0x) = \emptyseq(x, h(x))$.
Then, for any numeral $n$, we have the $\PV^-$-proof of $h(g(0 n)) = \emptyseq$, whose length is constant.
\begin{align}
  h(g(0n)) &= h(\overbrace{0 \cdots 0}^{k} g(n)) \\
  &= \emptyseq^2(\overbrace{0 \cdots 0}^{k-1} g(n), h(\overbrace{0 \cdots 0}^{k-1} g(n))) \label{eq:counter-example-rec}\\
  &= \emptyseq \label{eq:counter-example-epsilon}
\end{align}
However, the computation of $h(g(0n))$, which is defined in \cite{Yamagata2014}, becomes
\begin{equation}\label{eq:counter-example}
  \infer{\CompSt{h(g(0n))}{\rho}{v}}{
    \ldots &
    \CompSt{g(0 n)}{\rho}{v_0}
  }
\end{equation}
and the length of the computation of $g(0n)$ rapidly increases depending on $n$.
Because $\emptyseq$ can be computed by a computation with a constant length, this contradicts Proposition 1 of \cite{Yamagata2014}.

This indicates that there is a gap in the proof of \cite{Yamagata2014}.
Indeed, the computation of $\emptyseq(0 \cdots 0 g(n), h(0 \cdots 0 g(n)))$ does not have a form such as (92) in the proof of Lemma 14 in \cite{Yamagata2014}, because it contains computations neither for $0 \cdots 0 g(n)$ nor for $h(0 \cdots 0 g(n))$.

In this paper, we reformulate the computation rules such that their forms have greater uniformity.
Therefore, to compute $\emptyseq(0 \cdots 0 g(n), h(0 \cdots 0 g(n)))$, we need to compute $0 \cdots 0 g(n)$ and $h(0 \cdots 0 g(n))$.
Thus, Proposition \ref{prop:sound} holds for the equality \eqref{eq:counter-example-rec}.
However, to ensure that Proposition \ref{prop:sound} holds for the equality \eqref{eq:counter-example-epsilon}, we introduce approximate computations, in which the value can be approximated by $*$.
By evaluating $0 \cdots 0 g(n)$ and $h(0 \cdots 0 g(n))$ to $*$, the number of steps of the computation of $\emptyseq(0 \cdots 0 g(n), h(0 \cdots 0 g(n)))$ can be bounded by a constant.
Then, instead of \eqref{eq:counter-example}, we use the computation that has a constant size.
\begin{equation}
  \infer{\CompSt{h(g(0n))}{\rho}{\emptyseq}}{
    \infer{\CompSt{\emptyseq^2(x, h(y))}{[x, y \mapsto *]}{\emptyseq}}{
      \CompSt{x}{[x, y \mapsto *]}{*}&
      \CompSt{h(y)}{[x, y \mapsto *]}{*}
    }
     &
    \infer={\CompSt{g(0 n)}{\rho}{0 *}}{
      \infer={\CompSt{0 \cdots 0 x}{[x \mapsto *]}{0 *}}{} &
      \CompSt{g(n)}{\rho}{*}
    }
  }
\end{equation}
where $[x, y \mapsto *]$ denotes the environment that assigns $*$ to $x$ and $y$.
$[x \mapsto *]$ has a similar meaning.
Thus, Proposition \ref{prop:sound} holds for \eqref{eq:counter-example-epsilon}.

\subsection{Meta-theories}\label{subsec:meta}

In this paper, we strengthen the meta-theory from $S^1_2$, which is claimed to be sufficient in the previous draft \cite{Yamagata2014}, to $S^2_2$.
This is because the proof of Lemma \ref{lem:subst-I}, \ref{lem:subst-II} and Proposition \ref{prop:sound} requires $\Pi^b_2 - \PIND$.
The reason for Lemma \ref{lem:subst-I} and \ref{lem:subst-II} is that the conclusions of a computation used for induction step change and their number increases. The reasons for Proposition \ref{prop:sound} are the transitivity and substitution rules.
To interpret the transitivity rule, the induction hypothesis must hold for all computations with certain conditions.
Similarly, to interpret substitution, the induction hypothesis must hold for all environments with certain conditions.
Therefore, the induction hypothesis has universal quantifiers in the outmost position.
Further, the induction hypothesis claims that for each computation of the term in the left-hand side of the conclusion, there is a computation of the term in the right-hand side of the conclusion.
Therefore, induction hypothesis becomes $\Pi^b_2$.

\subsection{Relation to result of Buss and Ignjatovi\'c}\label{sec:buss-ignja}

This paper presents proof that Buss's $S^2_2$ is capable of proving the consistency of purely equational $\PV^-$, which is obtained by removing induction from $\PV$ of Cook and Urquhart but retaining the substitution rule.
Because Buss and Ignjatovi\'c stated that this is impossible in $S^1_2$, at first glance, it implies that $S^1_2 \subsetneq S^2_2$.
However, this is not the case.

Although they stated that $S^1_2$ cannot prove the consistency of purely equational $\PV^-$, what they actually prove is that $S^1_2$ cannot prove the consistency of $\PV^-$, which is extended by propositional logic and $BASIC^e$ axioms.
According to them, we can obtain the same unprovability for purely equational $\PV^-$ by translating propositional connectives into numerical functions.
For example, $t = u$ is translated into $\Eq(t, u)$, where the function $\Eq$ is defined as
\begin{equation}
  \Eq(t, u) = \begin{cases}
    0 & \text{if $t = u$}\\
    1 & \text{otherwise,}
\end{cases}
\end{equation}
and $p \vee q$ into $p \cdot q$, etc.
Then, every proposition $p$ is translated to a numerical term $t_p$.
Then, they assume that whenever a proposition $p$ is proved by $\PV^-$, which is extended with propositional logic and $BASIC^e$, $t_p = 0$ can be proved in purely equational $\PV^-$.
 However, although such translation is possible in $\PV$ \cite{Cook1993}, it depends on the existence of induction.
For example, the reflexive law $x = x$ is translated into $\Eq(x, x) = 0$.
It is impossible to derive the latter from the former without using induction.
Therefore, we cannot conclude that the consistency of $\PV^-$ with propositional logic and $BASIC^e$ axioms from the consistency of purely equational $\PV^-$ in $S^1_2$.
Thus, our result does not appear to imply that $S^1_2 \subsetneq S^2_2$.

One possible way to prove that $S^1_2 \subsetneq S_2$ would be to prove the consistency of $\PV^-$ with propositional logic and $BASIC^e$ axioms in $S_2$, which is the system considered by Buss and Ignjatovi\'c.
However, because our method relies on the fact that $\PV^-$ is formulated as an equational theory, our method cannot be extended to $\PV^-$ with propositional logic and $BASIC^e$ axioms.
Thus, as a long-term goal, it would be interesting to develop a technique to prove the consistency of such a system in bounded arithmetic.

\begin{ack}
  The author is grateful to Arnold Beckmann for pointing out a counter-example in the previous draft.
  The author is also grateful to Toshiyasu Arai, Satoru Kuroda, Jun Inoue, Izumi Takeuti, and Kazushige Terui for insightful discussions and comments.
  The comments of an anonymous referee in response to the previous submission were also very helpful in improving the paper.
  We would like to thank Editage (www.editage.jp) for English-language editing.
  This work was partially supported by the Research Institute for Mathematical Sciences, a Joint Usage/Research Center located in Kyoto University.
\end{ack}


\end{document}